\documentclass[a4paper,11pt,reqno]{article}
\usepackage[latin1]{inputenc}
\usepackage[english]{babel}
\usepackage{a4wide, fullpage}
\usepackage{amsthm}
\usepackage{amsmath}
\usepackage{amsfonts}
\usepackage{amssymb}
\usepackage{a4}
\usepackage{bbm}
\usepackage{bm}

\usepackage{graphicx}
\usepackage{subfigure}
\usepackage{fancybox}
\theoremstyle{plain}
\newtheorem{theorem}{Theorem}[section]
\newtheorem{proposition}[theorem]{Proposition}
\newtheorem{lemma}[theorem]{Lemma}

\newtheorem{corollary}[theorem]{Corollary}
\newtheorem{remark}[theorem]{Remark}

\theoremstyle{definition}
\newtheorem{example}[theorem]{Example}

\numberwithin{equation}{section}

\newcommand{\N}{\ensuremath{\mathbb{N}}}
\newcommand{\Z}{\ensuremath{\mathbb{Z}}}

\newcommand{\R}{\ensuremath{\mathbb{R}}}

\renewcommand{\P}{\ensuremath{\mathbb{P}}}
\newcommand{\E}{\ensuremath{\mathbb{E}}}
\newcommand{\1}{\ensuremath{\mathbbm{1}}}
\newcommand{\M}{\ensuremath{\mathcal{M}}}

\newcommand{\bzeta}{\ensuremath{\boldsymbol\zeta}}

\title{A modified lookdown construction for the Xi-Fleming-Viot
process with mutation and populations with recurrent bottlenecks}

\author{Matthias Birkner\footnote{Weierstra{\ss}-Institut f\"ur Angewandte
Analysis und Stochastik, Mohrenstra{\ss}e 39, 10117 Berlin, Germany,
e-mail: {\tt birkner@wias-berlin.de},
URL: {\tt http://www.wias-berlin.de/}$\sim${\tt birkner}}\ ,
Jochen Blath\footnote{Institut f\"ur Mathematik, Technische
Universit\"at Berlin, Stra{\ss}e des 17.~Juni 136, 10623 Berlin, Germany,
e-mail: {\tt blath@math.tu-berlin.de}, {\tt steinrue@math.tu-berlin.de},
{\tt johannatams@gmx.de},
URL: {\tt http://www.math.tu-berlin.de/}$\sim${\tt blath},
{\tt http://www.math.tu-berlin.de/}$\sim${\tt steinrue}}\ ,
Martin M\"ohle\footnote{Mathematisches Institut,
Universit\"at D\"usseldorf, Universit\"atsstra{\ss}e 1, 40225 D\"usseldorf,
Germany, e-mail: {\tt moehle@math.uni-duesseldorf.de},
URL: {\tt http://www.math.uni-duesseldorf.de/Personen/indiv/Moehle}
(corresponding author)}\ ,\\
Matthias Steinr\"ucken$^\dag$ and Johanna Tams$^\dag$}

\date{27th of October, 2008}

\begin{document}

\maketitle

\begin{abstract}
  Let $\Lambda$ be a finite measure on the unit interval. A
  $\Lambda$-Fleming-Viot process is a probability measure valued
  Markov process which is dual to a coalescent with multiple collisions
  ($\Lambda$-coalescent) in analogy to the duality known for
  the classical Fleming-Viot process and Kingman's coalescent, where
  $\Lambda$ is the Dirac measure in $0$.

  We explicitly construct a dual process of the coalescent with
  simultaneous multiple collisions ($\Xi$-coalescent) with mutation,
  the $\Xi$-Fleming-Viot process with mutation, and
  provide a representation based on the empirical measure of an
  exchangeable particle system along the lines of Donnelly and Kurtz
  (1999). We establish pathwise convergence of the approximating
  systems to the limiting $\Xi$-Fleming-Viot process with mutation.
  An alternative construction of the semigroup based on the
  Hille-Yosida theorem is provided and various types of duality of the
  processes are discussed.

  In the last part of the paper a population is considered which
  undergoes recurrent bottlenecks. In this scenario, non-trivial
  $\Xi$-Fleming-Viot processes naturally arise as limiting models.
\end{abstract}

AMS subject classification.
{\em Primary:}
               60K35;   % Interacting random processes
               60G09;   % Exchangeability
               92D10    % Genetics
{\em Secondary:} 60C05; % Combinatorial probability
                 92D15  % Problems related to evolution

\vspace{2mm}

lookdown construction

\vspace{2mm}

Running head: Lookdown construction for Fleming-Viot processes

%%\tableofcontents

\section{Introduction and main results}\label{sec:introduction}

\subsection{Motivation}
   One of the fundamental aims of mathematical population genetics is
   the construction of population models in order to describe and
   to analyse certain phenomena which are of interest for biological
   applications. Usually these models are constructed such that they
   describe the evolution of the population under consideration
   forwards in time. A classical and widely used model of this kind is
   the Wright-Fisher diffusion, which can be used for large populations
   to approximate the evolution of the fraction of individuals carrying a
   particular allele. On the other hand it is often quite helpful
   to look from the present back into the past and to trace back the ancestry
   of a sample of $n$ individuals, genes or particles. In many situations,
   the Kingman coalescent \cite{K82a, K82b} turns out to be an
   appropriate tool to approximate the ancestry of a sample
   taken from a large population. It is well known that the Wright-Fisher
   diffusion is dual to the block counting process of the Kingman
   coalescent \cite{D86, M01}. More general, the Fleming-Viot process
   \cite{FV79}, a measure-valued extension of the Wright-Fisher diffusion,
   is dual to the Kingman coalescent.
   Such and similar duality results are quite common in particular in the
   physics literature on interacting particle systems \cite{L85} and in
   the more theoretical literature on mathematical population genetics
   \cite{AH07, AS05, DK96, DK99, EK95, H00, M99, M01}. Donnelly and Kurtz
   \cite{DK96} established a so-called lookdown construction and used this
   construction to show that the Fleming-Viot process is dual to the Kingman
   coalescent. This construction and corresponding duality results have been
   extended \cite{DK99, BLG03, BLG05, BLG06} to the $\Lambda$-Fleming-Viot
   process, which is the measure-valued dual of a coalescent process
   allowing for multiple collisions of ancestral lineages. For more
   information on coalescent processes with multiple collisions, so-called
   $\Lambda$-coalescents, we refer to Pitman \cite{P99} and Sagitov \cite{S99}.

   There exists a broader class of coalescent processes
   \cite{MS01, S00, S03} in which many multiple
   collisions can occur with positive probability simultaneously at the same
   time. These processes can be characterized by
   a measure $\Xi$ on an infinite simplex and are hence called
   $\Xi$-coalescents. It is natural to further extend the above
   constructions and results to this full class of coalescent
   processes and, in particular, to provide constructions of the dual processes,
   called $\Xi$-Fleming-Viot processes. Although such
   extensions have been briefly indicated in \cite{DK99} and
   \cite{BLG03}, these extensions have not been carried out in detail yet.
   $\Xi$-coalescents have also recently been applied to study population genetic
   problems, see \cite{TV08, SW08}.

   The motivation to present this paper is hence manifold. We explicitly construct
   the $\Xi$-Fleming-Viot process and provide a representation via empirical
   measures of an exchangeable particle system in the spirit of Donnelly
   and Kurtz \cite{DK96, DK99}. We furthermore establish corresponding convergence
   results and pathwise duality to the $\Xi$-coalescent. We also provide
   an alternative, more classical functional-analytic construction of the
   $\Xi$-Fleming-Viot process based on the Hille-Yosida theorem and present
   representations for the generator of the $\Xi$-Fleming-Viot process. Our
   approaches include neutral mutations. The results give insights into the
   pathwise structure of the $\Xi$-Fleming-Viot process and its dual
   $\Xi$-coalescent. Examples and situations are presented in which certain
   $\Xi$-Fleming-Viot processes and their dual $\Xi$-coalescents
   occur naturally.

\subsection{Moran models with (occasionally) large families}\label{GenMoranMod}
Consider a population of fixed size $N\in{\mathbb N}:=\{1,2,\ldots\}$ and
assume that each individual is of a certain type, where the space $E$ of
possible types is assumed to be compact and Polish. Furthermore
assume that for each
vector ${\bf k}=(k_1,k_2,\ldots)$ of integers satisfying $k_1\ge k_2\ge\cdots\ge 0$
and $\sum_{i=1}^\infty k_i\le N$ a non-negative real quantity $r_N({\bf k})\ge 0$
is given. The population is assumed to evolve in continuous time as follows.
Given a vector ${\bf k}=(k_1,\ldots,k_m,0,0,\ldots)$, where
$k_1\ge \cdots\ge k_m\ge 1$ and $k_1+\cdots+k_m\le N$, with rate $r_N({\bf k})$ we
choose randomly $m$ groups of sizes $k_1,\ldots,k_m$ from the present population.
Inside each of these $m$ groups we furthermore choose randomly a `parent'
which forces all individuals in its group to change their type to the type of
that parent. We say that a ${\bf k}$-reproduction event occurs with rate
$r_N({\bf k})$. The classical Moran model corresponds to $r_N(2,0,0,\ldots)=N$.

Except for the fact that these models are formulated in continuous time, they
essentially coincide with the class of neutral exchangeable population models
with non-overlapping generations introduced by Cannings \cite{C74, C75}.
Starting with the seminal work of Kingman \cite{K82a, K82b}, the genealogy
of samples taken from such populations is well understood, in particular
for the situation when the total population size $N$ tends to infinity.

\subsection{Genealogies and exchangeable coalescents}\label{ExCoal}
For neutral population models of large, but fixed population size and
finite-variance reproduction mechanism, Kingman \cite{K82a} showed that
the genealogy of a finite sample of size $n$ can be approximately
described by the so called $n$-coalescent
$(\Pi^{\delta_{\bf 0},(n)}_t)_{t\ge 0}$.  The $n$-coalescent is a
time-homogeneous Markov process taking values in $\mathcal{P}_n$,
the set of partitions of $\{1,\ldots,n\}$. If $i$ and $j$ are in the
same block of the
partition $\Pi^{\delta_{\bf 0},(n)}_t$, then they have a common
ancestor at time $t$ ago.  $\Pi^{\delta_{\bf 0},(n)}_0$ is the
partition of $\{1,\ldots,n\}$ into singleton blocks.  The transitions
are then given as follows: If there are $b$ blocks at present, then
each pair of blocks merges with rate 1, thus the overall rate of
seeing a merging event is ${b \choose 2}$. Note that only binary
mergers are allowed and that at some random time, all individuals will
have a (most recent) common ancestor.

Kingman \cite{K82a} also showed that there exists a
$\mathcal{P}_\N$-valued Markov process $(\Pi^{\delta_{\bf 0}}_t)_{t\ge 0}$,
where $\mathcal{P}_\N$ denotes the set of partitions of $\N$. This process,
the so-called Kingman coalescent, is
characterised by the fact that for each $n$ the restriction of
$(\Pi^{\delta_{\bf 0}}_t)_{t\ge 0}$ to the first $n$ natural numbers is the
$n$-coalescent. The process can be constructed by an application of
the standard Kolmogoroff extension theorem, since the restriction of every
$n$-coalescent to $\{1,\ldots,m\}$, where $1\le m\le n$, is an
$m$-coalescent.

Whereas the Kingman coalescent allows only for binary mergers, the
idea of a time-homo\-ge\-neous $\mathcal{P}_\N$-valued Markov process that
evolves by the coalescence of blocks was extended by Pitman
\cite{P99} and Sagitov \cite{S99} to coalescents where multiple blocks
are allowed to merge at the same time, so-called $\Lambda$-coalescents, which
arise as the limiting genealogy of populations where the variance of
the offspring distribution diverges as the population size tends to
infinity. M\"ohle and Sagitov \cite{MS01} and Schweinsberg \cite{S00}
introduced the even larger class of coalescents with simultaneous
multiple collisions, also called exchangeable coalescents or
$\Xi$-coalescents, which describe the genealogies of populations
allowing for large family sizes.

Schweinsberg \cite{S00} showed that any exchangeable coalescent
$(\Pi^\Xi_t)_{t\ge 0}$ is characterised by a finite measure
$\Xi$ on the infinite simplex
\[
\Delta\ :=\ \{\bzeta=(\zeta_1,\zeta_2,\ldots):\zeta_1\ge\zeta_2\ge\cdots\ge 0,
\mbox{$\sum_{i=1}^\infty$}\zeta_i\le 1\}.
\]
Throughout the paper, for $\bzeta\in\Delta$, the notation
$|\bzeta|:=\sum_{i=1}^\infty\zeta_i$ and $(\bzeta,\bzeta):=\sum_{i=1}^\infty
\zeta_i^2$ will be used for convenience.
Note that M\"ohle and Sagitov \cite{MS01} provide an alternative (though
somewhat less intuitive) characterisation of the $\Xi$-coalescent
based on a sequence of finite symmetric measures $(F_r)_{r\in\N}$.
Coalescent processes with multiple collisions ($\Lambda$-coalescents)
occur if the measure $\Xi$ is concentrated on the subset of all points
$\bzeta\in\Delta$ satisfying $\zeta_i=0$ for all $i\ge 2$. The
Kingman-coalescent corresponds to the case $\Xi = \delta_{\bf 0}$.
It is convenient to decompose the measure $\Xi$ into a `Kingman part'
and a `simultaneous multiple collision part', that is,
$\Xi=a\delta_{\mathbf 0}+\Xi_0$ with $a:=\Xi(\{\mathbf 0\})\in[0,\infty)$
and $\Xi_0(\{\mathbf 0\})=0$. The transition rates of the $\Xi$-coalescent
$\Pi^\Xi$ are given as follows. Suppose there are currently $b$ blocks.
Exactly $\sum_{i=1}^r k_i$ blocks collide into $r$ new blocks, each
containing $k_1,\dots,k_r\ge 2$ original blocks, and
$s$ single blocks remain unchanged, such that the condition
$\sum_{i=1}^r k_i+s=b$ holds. The order of $k_1,\dots,k_r$ does not matter.
The rate at which the above collision happens is then given as
(Schweinsberg \cite[Theorem~2]{S00})
\begin{equation} \label{rates}
   \lambda_{b;k_1,\dots,k_r;s}
   \ =\ a\1_{\{r=1,k_1=2\}} +
   \int_\Delta\sum_{l=0}^s {s\choose l}(1-|\bzeta|)^{s-l}
   \sum_{i_1\neq\cdots\neq i_{r+l}}
     \zeta_{i_1}^{k_1}\cdots\zeta_{i_r}^{k_r}\zeta_{i_{r+1}}\cdots\zeta_{i_{r+l}}
   \frac{\Xi_0(d{\bzeta})}{(\bzeta,\bzeta)}.
\end{equation}
An intuitive explanation of (\ref{rates}) is given below in terms
of Schweinsberg's \cite{S00} Poisson process construction of the
$\Xi$-coalescent. If $\Xi(\Delta)\ne 0$, then without loss of
generality it can be assumed that $\Xi$ is a probability measure,
as remarked after Eq. (12) of \cite{S00}. Otherwise simply divide
each rate by the total mass $\Xi(\Delta)$ of $\Xi$.

\subsection{Poisson process construction of the $\Xi$-coalescent}\label{Erhard}
It is convenient to give an explicit construction of the $\Xi$-coalescent
in terms of Poisson processes. Indeed, Schweinsberg \cite[Section 3]{S00}
shows that the $\Xi$-coalescent can be constructed from a family of Poisson
processes $\{\mathfrak{N}^K_{i,j}\}_{i,j\in\N,i<j}$ and a Poisson point
process $\mathfrak{M}^{\Xi_{0}}$ on $\R_+\times\Delta\times [0,1]^{\N}$.
The processes $\mathfrak{N}^K_{ij}$ have rate $a=\Xi(\{\mathbf 0\})$
each and govern the binary mergers of the coalescent. The process
$\mathfrak{M}^{\Xi_{0}}$ has intensity measure
\begin{equation}
\label{intensity_measure}
   dt\otimes\frac{\Xi_0(d\bzeta)}{(\bzeta,\bzeta)}
   \otimes(\1_{[0,1]}(t)dt)^{\otimes\N}.
\end{equation}
These processes can be used to construct the $\Xi$-coalescent as
follows: Assume that before the time $t_m$ the process $\Pi$ is
in a state $\{B_1,B_2,\ldots\}$. If $t_m$ is a point of increase
of one of the processes $\mathfrak{N}^K_{i,j}$ (and there are at
least $i\vee j$ blocks), then we merge the corresponding blocks
$B_i$ and $B_j$ into a single block (and renumber). This mechanism
corresponds to the Kingman-component of the coalescent.

The non-Kingman collisions are governed by the points
\begin{equation}
    (t_m,{\bzeta}_m,{\bf u}_m)
    \ =\ (t_m,(\zeta_{m1},\zeta_{m2},\ldots),(u_{m1},u_{m2},\ldots))
\end{equation}
of the Poisson process $\mathfrak{M}^{\Xi_0}$. The random vector
${\bzeta}_m$ denotes the respective asymptotic family sizes in the
multiple merger event at time $t_m$ and the ${\bf u}_m$ are ``uniform
coins'', determining the blocks participating in the respective merger
groups; see \eqref{eq_definition_g} or \cite[Section 3]{S00} for details.

\subsection{$\Xi$-Fleming-Viot processes}

An in many senses dual approach to population genetics is to view a
population of finite size as a vector of types $(Y^N_1,\ldots,Y^N_N)$
with values in $E^N$ or as an empirical measure of that vector
$\frac{1}{N}\sum_{i=1}^N\delta_{Y^N_i}$ and look at the evolution
under mutation and resampling forwards in time. When $N$ tends to
infinity one obtains the Fleming-Viot process \cite{FV79}. This
process has been extended to incorporate other important biological
phenomena and has found wide applications, see \cite{EK93} for a survey.

Donnelly and Kurtz \cite{DK96} embedded an $E^\infty$-valued particle
system into the classical Fleming-Viot process, via a clever lookdown
construction, and showed that it is dual to the Kingman-coalescent.
This construction and the duality has been extended to the so-called
$\Lambda$-Fleming-Viot process, dual to the
$\Lambda$-coalescents, and investigated by several authors, see, e.g.,
\cite{DK99, BBC05, BLG03, BLG05, BLG06}, or \cite{BB07} for an
overview.

Let $f\in C_b(E^p)$, $\mu \in \M_1(E)$ and $G_f(\mu) := \langle f ,
\mu^{\otimes p} \rangle$. The generator of the $\Lambda$-Fleming-Viot
process without mutation has the form (see \cite[Equation (1.11)]{BBC05})
\begin{equation}\label{eq_lambda_fleming_viot}
    L^{\Lambda}G_f(\mu) = \sum_{J\subset\{1,\ldots,p\},|J|\ge 2} \lambda_{p;|J|;p-|J|} \int \big(f({\bf x}^J) - f({\bf x})\big)\,\mu^{\otimes p}(d{\bf x}),
\end{equation}
where
\begin{equation}
   ({\bf x}^J)_i\ =\
   \begin{cases}
      x_{\min(J)} & \mbox{if $i \in J$,}\\
      x_i         & \mbox{otherwise.}
   \end{cases}
\end{equation}
Note that (\ref{eq_lambda_fleming_viot}) includes the generator of the
classical Fleming-Viot process (without mutation) if the summation is
restricted to sets $J$ satisfying $|J|=2$.

Our aim in this paper is to present the modified lookdown construction for a
measure-valued process that we will call the $\Xi$-Fleming-Viot process with
mutation, or the $(\Xi,B)$-Fleming-Viot process. The symbol $B$ stands here
for an operator describing the mutation process. We will establish its
duality to the $\Xi$-coalescent with mutation.
The modified lookdown construction will also enable us to establish some path
properties of the $(\Xi,B)$-Fleming-Viot process.

\subsection{A modified lookdown construction of the $(\Xi,B)$-Fleming-Viot process}

Consider a population described by a vector $Y^N(t)=(Y^N_1(t),\ldots,Y^N_N(t))$
with values in $E^N$, where $Y^N_i(t)$ is the type of individual $i$ at
time $t$. The evolution of this population (forwards in time) has two
components, namely reproduction and mutation. During its lifetime, each
particle undergoes mutation according to the bounded linear mutation operator
\begin{equation} \label{mutop}
   Bf(x)\ =\ r\int_E(f(y)-f(x))\,q(x,dy),
\end{equation}
where $f$ is a bounded function on $E$, $q(x,dy)$ is a Feller transition
function on $E\times\mathcal B(E)$, and $r\ge 0$ is the global mutation rate.

The resampling of the population is governed by the Poisson point
process $\mathfrak{M}^{\Xi_{0}}$, which was introduced as a driving
process for the $\Xi$-coalescent. In particular, the resampling events
allow for the simultaneous occurrence of one or more large families.
The resampling procedure is described in detail in Section~\ref{exchangeable}.
An important fact is that this resampling is made such that it retains
exchangeability of the population vector.

In Section~\ref{exchangeable}, we introduce another particle system
$X^N=(X^N_1,\ldots,X^N_N)$ again with values in $E^N$. Each particle
mutates according
to the same generator \eqref{mutop} as before. For the resampling
event, we will use the same driving Poisson point process
$\mathfrak{M}^{\Xi_{0}}$, but we will use the modified lookdown
construction of Donnelly and Kurtz introduced in \cite{DK99}, suitably
adapted to our scenario. This $(\Xi,B)$-lookdown process will be
introduced in Section~\ref{orderedmodel}. It is crucial that the
resampling events retain exchangeability of the population vector
and that the process $\{X^N(t)\}$ has the same empirical measure $\sum_{i=1}^N
\delta_{X^N_i(t)}$ as the process $\{Y^N(t)\}$.

The construction of the resampling events allows us to pass to the
limit as $N$ tends to infinity and obtain an $E^\infty$-valued
particle system $X=(X_1,X_2,\ldots)$. Since this particle system is
also exchangeable, this procedure enables us to access the almost sure limit of
the empirical measure as $N$ tends to infinity by the De Finetti
Theorem (which is not possible for the $Y^N$).
\subsection{Results} \label{suse_results}
Let $\mathcal{D}(B)$ denote the domain of the mutation generator $B$ and
let $f_1,f_2,\ldots\in\mathcal{D}(B)$ be functions that separate points
of $\M_1(E)$ in the sense that $\int f_k\,d\mu=\int f_k\,d\nu$ for all
$k\in\N$ implies that $\mu=\nu$. Such sequences exist, see, e.g.\ Section 1
(Lemma 1.1 in particular) of \cite{DK96}. We use the metric $d$ on
$\M_1(E)$ defined via
\begin{equation}\label{metric_on_measures}
   d(\mu,\nu)\ :=\ \sum_k\frac{1}{2^k}
   \Big|\int f_k\,d\mu-\int f_k\,d\nu\Big|,
   \qquad\mu,\nu\in\M_1(E)
\end{equation}
and equip the topology of locally uniform convergence on
$D_{\M_1(E)}([0,\infty))$ with the metric
\begin{equation}\label{metric_on_paths_of_measures}
   d_p(\mu,\nu)\ :=\ \int_0^\infty e^{-t}d(\mu(t),\nu(t))\,dt.
\end{equation}
\begin{theorem} \label{main}
   The $\mathcal M_1(E)$-valued process $(Z_t)_{t\ge 0}$, defined in terms
   of the {\em ordered} particle system $X=(X^1,X^2,\dots)$ by
   \[
   Z_t\ :=\ \lim_{n\to\infty}Z^n_t
   \ =\ \lim_{n\to\infty}\frac 1n \sum_{i=1}^n \delta_{X_i(t)},\quad t\ge 0,
   \]
   is called the {\em $\Xi$-Fleming-Viot process with mutation operator $B$}
   or simply the {\em $(\Xi,B)$-Fleming-Viot process}. Moreover, the
   empirical processes $(Z_t^n)_{t\ge 0}$ converge almost surely on the
   path space $D_{\mathcal M_1(E)}([0,\infty))$ to the c\`adl\`ag process
   $(Z_t)_{t\ge 0}$.
\end{theorem}
Since the empirical measures of $X^N$ and $Y^N$ are identical, we arrive
at the following corollary.
\begin{corollary} \label{maincor}
   Define, for each $n$,
   \[
   \tilde Z^n_t:= \frac 1n \sum_{i=1}^n \delta_{Y_i(t)}, \quad t\ge 0,
   \]
   the empirical process of the $n$-th unordered particle system, and
   assume that $\tilde Z_0^n \to Z_0$ weakly as $n\to\infty$. Then,
   $(\tilde Z_t^n)_{t\ge 0}$ converges weakly on the path space
   $D_{\mathcal M_1(E)}([0,\infty))$ to the $(\Xi,B)$-Fleming-Viot
   process $(Z_t)_{t\ge 0}$.
\end{corollary}
The Markov process $(Z_t)_{t\ge 0}$ is characterized by its generator
as follows.
\begin{proposition} \label{mainprop}
   The $(\Xi,B)$-Fleming-Viot process $(Z_t)_{t\ge 0}$ is a strong
   Markov process. Its generator, denoted by $L$, acts on test functions
   of the form
   \begin{equation} \label{eq:testfunction}
      G_f(\mu)\ :=\ \int_{E^n} f(x_1,\dots,x_n)\,\mu^{\otimes n}(dx_1,\dots,dx_n),
      \quad \mu \in \M_1(E),
   \end{equation}
   where $f:E^n\to\R$ is bounded and measurable, via
   \begin{equation}\label{eq_genXiFV}
      L G_f(\mu)\ :=\
      L^{a\delta_{\bf 0}}G_f(\mu) + L^{\Xi_{0}}G_f(\mu) + L^B G_f(\mu),
   \end{equation}
   where
   \begin{align}
      L^{a\delta_{\bf 0}} G_f(\mu)
      & :=  a\sum_{1\le i < j \le n} \int_{E^n}
            \Big(
               f(x_1,\!.., x_i,\!.., x_i,\!.., x_n)- f(x_1,\!.., x_i,\!.., x_j,\!.., x_n)
            \Big)\mu^{\otimes n}(d{\bf x}),\label{eq_genXiFV_kingman}\\
      L^{\Xi_{0}} G_f(\mu)
      & := \int_\Delta \int_{E^\N}
                   \left[ G_f\big( (1-|\bzeta|)\mu + {\textstyle\sum_{i=1}^\infty \zeta_i \delta_{x_i}} \big)
                      - G_f(\mu) \right] \mu^{\otimes \N}(d\mathbf{x})
                   \frac{\Xi_{0}(d\bzeta)}{(\bzeta,\bzeta)}, \label{eq_genXiFV_xi}\\
      L^{B} G_f(\mu)
      & := r\sum_{i=1}^n \int_{E^n} B_i(f(x_1, \dots, x_n)) \mu^{\otimes n}(d{\bf x}),\label{eq_genXiFV_mutation}
   \end{align}
   and $B_i f$ is the mutation operator $B$, defined in (\ref{mutop}),
   acting on the $i$-th coordinate of $f$.
\end{proposition}

\begin{remark}{\rm
1) In \cite{DK99}, Donnelly \& Kurtz established a construction and
pathwise duality for the $\Lambda$-Fleming-Viot process. In some
sense, their paper works under the general assumption ``allow
simultaneous and/or multiple births and deaths, but we assume that all
the births that happen simultaneously come from the same parent'' (p.~166),
even though they very briefly in Section~2.5 mention a possible extension
to scenarios with simultaneous multiple births to multiple parents.
In essence, the present paper converts these ideas into theorems.

2) Note that in a similar direction, Bertoin \& Le Gall remark briefly on p.~277 of
\cite{BLG03} how their construction of the $\Lambda$-Fleming-Viot process via
flows of bridges can be extended to the simultaneous multiple merger context
(but leave details to the interested reader). We are not following this
approach, as it is hard to combine with a general type space and general mutation
process.

3) The $\Xi$-Fleming-Viot process has recently been independently
   constructed by Taylor and V\'eber (personal communication, 2008)
   via Bertoin and Le Gall's flow
   of bridges (see \cite{BLG03}) and Kurtz and Rodriguez' Poisson
   representation of measure-valued branching processes (see
   \cite{KR08}). In this context we refer to Taylor and V\'eber \cite{TV08}
   for a larger study of structured populations, in which $\Xi$-coalescents
   appear under certain limiting scenarios.

4) Note that the modified lookdown construction of the $\Lambda$-Fleming-Viot
   process contains all information available about the genealogy of the
   process and therefore also provides a pathwise embedding of the
   {\em $\Lambda$-coalescent measure tree} considered by Greven,
   Pfaffelhuber and Winter \cite{GPW07}. A similar statement holds for the $\Xi$-coalescent.}
\end{remark}

The rest of the paper is organised as follows: In
Section~\ref{exchangeable} we use the Poisson point process
$\mathfrak{M}^{\Xi_0}$ to introduce the finite unordered
$(\Xi,B)$-Moran model $Y^N$ and the finite ordered $(\Xi,B)$-lookdown
model $X^N$. It is shown that the ordered model is constructed in such
a way that we can let $N$ tend to infinity and obtain a well defined
limit. We will also show that the reordering preserves the
exchangeability property, which will be crucial for the proof in
Section~\ref{tightness}. In this section, we will introduce the
empirical measures of the process $Y^N$ and $X^N$, show that they are
identical and converge to a limiting process having nice path
properties, which is the statement of Theorem~\ref{main}.

Section~\ref{markov_semigroup} will be concerned with the generator of
the $\Xi_{0}$-Fleming-Viot process. We will give two alternative
representations and show that it generates a strongly continuous
Feller semigroup. Furthermore, we will show that the process
constructed in Section~\ref{tightness} solves the martingale problem
for this generator.

One representation of the generator will then be used in
Section~\ref{dualities} to establish a functional duality between the
$\Xi$-coalescent and the $\Xi$-Fleming-Viot process on the
genealogical level. Due to the Poissonian construction, this duality
can also be extended to a ``pathwise'' duality. We will also give a
function-valued dual, which incorporates mutation.

In Section~\ref{examples}, we look at two examples: The first example
is concerned with a population model with recurrent bottlenecks. Here,
a particular $\Xi$-coalescent, which is a subordination of Kingman's
coalescent, arises as a natural limit of the genealogical process. The
second example discusses the Poisson-Dirichlet-coalescent and obtains
explicit expressions for some quantities of interest.

\section{Exchangeable $E^\infty$-valued particle systems}\label{exchangeable}

\subsection{The canonical $(\Xi,B)$-Moran model}\label{canon_moran_modell}

We can use the Poisson process from Section~\ref{Erhard} governing the
$\Xi$-coalescent to describe a corresponding forward population model
in a canonical way, simply reversing the construction of the coalescent
by interpreting the merging events as birth events.

Consider the points
\begin{equation}
    (t_m,{\bzeta}_m,{\bf u}_m)\ =\
    (t_m,(\zeta_{m1},\zeta_{m2},\ldots),(u_{m1},u_{m2},\ldots))
\end{equation}
of $\mathfrak{M}^{\Xi_0}$ defined by (\ref{intensity_measure}).
The $t_m$ denote the times of reproduction events. Define
\begin{equation} \label{eq_definition_g}
   g(\bzeta,u)\ :=\
   \begin{cases}
      \min\{j\,|\,\zeta_1+\cdots+\zeta_j\ge u\}
             & \text{if }u \le\sum_{i\in\N}\zeta_i,\\
      \infty & \text{else}.
   \end{cases}
\end{equation}
At time $t_m$, the $N$ particles are grouped according to the values
$g({\bzeta}_m, u_{ml})$, $l=1,\dots,N$ as follows: For each $k \in \N$,
all particles $l \in \{1,\dots,N\}$ with $g({\bzeta}_m, u_{ml})=k$ form a family.
Among each non-trivial family we uniformly pick a `parent' and change the others'
types accordingly. Note that although the jump times $(t_m)$ may be dense in
$\R_+$, the condition
\[
\int_\Delta \sum_i\zeta_i^2\,\frac{\Xi_0(d{\bzeta})}{({\bzeta},{\bzeta})}
\ =\ \Xi(\Delta)\ <\ \infty
\]
guarantees that in a finite population, in each finite time interval only
finitely many non-trivial reproduction events occur. As above, each particle
follows an independent mutation process, according to \eqref{mutop},
in between reproductive events.

We describe the population corresponding to the $N$-particle $(\Xi,B)$-Moran
model at time $t\ge 0$ by a random vector
\begin{equation}
   Y^N(t)\ :=\ (Y^N_1(t),\ldots,Y^N_N(t))
\end{equation}
taking values in $E^N$.

\begin{remark}
  Note that this model is completely symmetric, thus, for each $t$, the
  population vector $Y^N(t)$ is exchangeable if $Y^N(0)$ is
  exchangeable.
\end{remark}

\subsection{The ordered model and exchangeability}
\label{orderedmodel}

We now define an ordered population model with the same family size
distribution, extending the ideas of Donnelly and Kurtz \cite{DK99} in
an obvious way.  This time each particle will be attached a ``level''
from $\{1,2,\dots\}$ in such a way that we obtain a nested coupling of
approximating $(\Xi,B)$-Moran models as $N$ tends to infinity. It will
be crucial to show that this ordered model retains initial
exchangeability, so that the limit as $N\to\infty$ of the empirical
measures of the particle systems, at each fixed time, exists by De
Finetti's theorem.

We will refer to this model as the the $(\Xi,B)$-lookdown-model. If
the population size is $N$, it will be described at time $t$ by the
$E^N$-valued random vector
\begin{equation}
\label{loddef}
    X^N(t)\ :=\ (X^N_1(t),\ldots,X^N_N(t)).
\end{equation}
The dynamics works as in the $(\Xi,B)$-Moran model above, including
the distribution of family sizes and the mutation processes for each
particle.

In each reproduction step, for each family, a ``parental'' particle will
be chosen, that then superimposes its type upon its family. This time,
however, the parental particle will not be chosen uniformly among the
members of each family (as in the $(\Xi,B)$-Moran model). Instead,
the parental particle will always be the particle with the lowest
level among the members of a family (hence each family member ``looks
down'' to their relative with the lowest level).
The attachment of types to levels is then rearranged as follows (see
Figure~\ref{fig_reproduction} for an illustration):
\label{items:atoc}
\begin{itemize}
\item[a)] All parental particles of all families (including the trivial ones) will retain their type and level.
\item[b)] All levels of members of families will assume the type of their respective parental particle.
\item[c)] All levels which are still vacant will assume the pre-reproduction types of non-parental particles
      retaining their initial order. Once all $N$ levels are filled, the remaining types will be lost.
\end{itemize}

\begin{figure}
\begin{center}
    \setcounter{subfigure}{0}
    \subfigure[Parental particles retain type and level.]{
%       \begin{minipage}{.3\textwidth}
        \centering\includegraphics[scale=1.2]{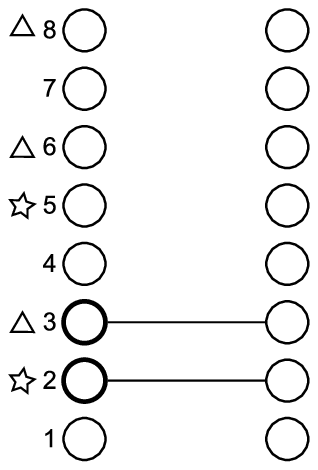}
%       \end{minipage}
    }
    \qquad
    \subfigure[Family members copy type of parental particle.]{
%       \begin{minipage}{.3\textwidth}
        \centering\includegraphics[scale=1.2]{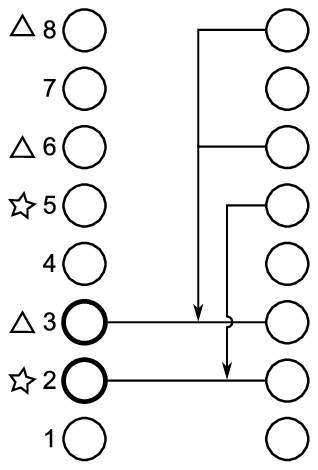}
%       \end{minipage}
    }
    \qquad
    \subfigure[Remaining particles retain their order and surplus particles get killed.]{
%       \begin{minipage}{.3\textwidth}
        \centering\includegraphics[scale=1.2]{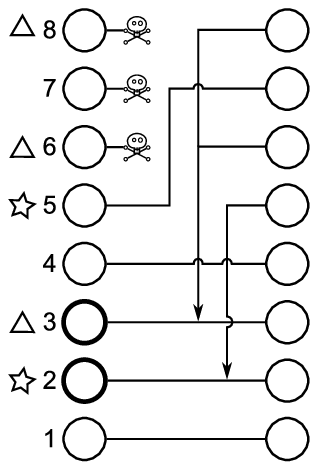}
%       \end{minipage}
    }
    \caption{An illustration of the reproduction mechanism in the
      $(\Xi,B)$-lookdown model. The particles at levels 2 and 5 belong
      to the ``star'' family, whereas the particles at levels 3, 6 and 8
      belong to the ``triangle'' family. The particles on the remaining
      levels belong to no family.  }\label{fig_reproduction}
\end{center}
\end{figure}

In this way, the dynamics of a
particle, at level $l$, say, will only depend on the dynamics of the
particles with {\em lower} levels. This consistency property allows to
construct all approximating particle systems, as well as their limit
as $N \to \infty$, {\em on the same probability space.}

Exchangeability of the modified $(\Xi,B)$-lookdown model is crucial in
order to pass to the De Finetti limit of the associated empirical
particle systems.  For each $N$, we will show that if $X(0)$ is
exchangeable, then $X$ is exchangeable at fixed times and at stopping
times.  The proof will rely on an explicit construction of uniform
random permutations $\Theta(t)$ which maps $X^N$ to $Y^N$.
\begin{theorem}\label{exchangeable_fixed_times}
  If the initial distribution of the population vector
  $(X^N_1(0),\dots,X^N_N(0))$ in the $(\Xi,B)$-lookdown-model is
  exchangeable, then $(X^N_1(t),\dots,X^N_N(t))$ is exchangeable for
  each $t\ge 0$.
\end{theorem}

For the rest of this section, we omit the superscript $N$ for the population models
in an attempt not to get \emph{lost in notation}.

The proof of Theorem~\ref{exchangeable_fixed_times} follows that of
Theorem~3.2 in \cite{DK99}.
We will construct a
coupling via a permutation-valued process $\Theta(t)$ such that
\begin{equation}\label{models_equal1}
   (Y_{1}(t),\dots,Y_N(t))=(X_{\Theta_{1}(t)}(t),\dots,X_{\Theta_N(t)}(t))
\end{equation}
and $\Theta(t)$ is uniformly distributed on all permutations of $\{1,\dots,N\}$
for each $t$ and independent of the empirical process up to time $t$ and
the ``demographic information'' in the model (see \eqref{eq_xi_filtration}
for a precise definition).

It suffices to construct the skeleton chain
$(\theta_m)_{m \in \N_0}$ of $\Theta$.
As a guide through the following notation, we have found it useful to occasionally
remember that $\Theta(t)$ (and its skeleton chain) is built to the following aim:
{\em \begin{center}
\begin{minipage}{.85\textwidth}\begin{center}
     $\Theta$ maps a position of an individual in the vector $Y$
     ($(\Xi,B)$-Moran-model) to the level of the corresponding individual in the
     ordered vector $X$ ($(\Xi,B)$-lookdown-model).\end{center}
   \end{minipage}
\end{center}}

\paragraph{Notation and ingredients}
For $N>0$ let $S_{N}$ denote the collection of all permutations
of $\{1,\dots,N\}$, let $P_{N}=\mathcal{P} (\{1,\dots,N\})$,
the set of all subsets of $\{1,\dots,N\}$, and
let $P_{N,k}\subset P_{N}$  be the subcollection of subsets with
cardinality $k$. For a set $M$, $M(i)$  will denote the $i$th
largest element in $M$.

At time $m$ (for the skeleton chain) let $c_m$ the total number of children.
Let $a_m$ be the number of families and $c_m^i$ the number of children born
to family $i$, hence
\begin{equation}\label{sum_km}
   \sum_{i=1}^{a_m} c^i_m = c_m.
\end{equation}

Note that we allow $c_m^i = 0$ for some, but not all $i$. These are
the trivial families where only the parental particle is below level $N$ and all
potential children are above.  Furthermore, we
need to keep track of these ``one-member families'' in order to match the
rates of our model to those of the $\Xi$-coalescent later on.

Let $\theta_{0}$ be uniformly distributed over $S_N$.
For each $m \in \N$, pick (independently, and independent of $\theta_0$)
\begin{itemize}
\item $\Phi_{m}$ a random set, uniformly chosen from $P_{N,c_m+a_m}$,
\item $\big(\phi_{m}^{1},\dots,\phi_{m}^{a_{m}}\big)$ a random (ordered) partition of
$\Phi_m$, such that each $\phi_m^i$ has size $c_m^i+1$,
%%$i=1,\dots,a_m$,
%%uniformly distributed on all partitions satisfying these restrictions,
\item $\sigma_m^i$, $i=1,\dots,a_m$ random permutations, each
$\sigma_{m}^i$ uniformly distributed over $S_{c_m^i + 1}$,
independently of $\Phi_m$ and the $\phi_m^i$.
\end{itemize}
Denote
\begin{itemize}
\item $\mu_{m}^{i}:=\min\phi_{m}^{i}, \; i \in \{1,\dots,a_{m}\}$, \;\;\mbox{and}
\item write $\Delta_m$ for the set of the highest $c_m$ integers from
  $\{1,\ldots,N\}\setminus\bigcup_{i=1}^{a_m}\mu_m^i$.
\end{itemize}

Proceeding inductively we assume that $\theta_{m-1}$ has already been
defined. We then construct $\theta_m$ as follows: Let
\begin{itemize}
\item $\nu_{m}^{i}:=\theta_{m-1}^{-1}(\mu_{m}^{i})$,
\item $\psi_m := \theta_{m-1}^{-1}(\Delta_m), \quad \mbox{and}$
\item a random ordered ``partition'' $\big(\psi_{m}^{1},\dots,\psi_{m}^{a_{m}}\big)$
of $\psi_m$ such that $|\psi_m^i|=c^i_m$, chosen independently of everything else.
\end{itemize}

In view of our intended application of $\theta_m$ to transfer from the
Moran model to the lookdown model, we will later on interpret these
quantities as follows:
In the $m$-th event, $\mu_m^i$ will be the level of the parental particle of family
$i$ in the lookdown-model, and $\nu_m^i$ will be the corresponding
index in the (unordered) Moran model.
$\Delta_m$ will specify the levels in the lookdown-model at which individuals die.
We do not just pick the highest $c_m$ levels, because we wish to retain parental particles.
$\psi_m$ will be the corresponding indices in the Moran model.
$\big(\phi_{m}^{1},\dots,\phi_{m}^{a_{m}}\big)$ describes the family decomposition
(including the respective parents) in this event in the lookdown model,
and
$\psi_m^i$ are the indices of the children in the $i$-th family in the Moran model.
Thus, $\theta_m$ will map $\phi_m^i$ to $\psi_m^i \cup \{\nu_m^i\}$ (in a particular
order).
\smallskip

Finally, define $\theta_m$ as follows:
Put $\Psi_m := \{\nu_m^1,\dots,\nu_m^{a_m}\} \cup \psi_m$. On $\Psi_m$,
\begin{equation}\label{construction_theta_m_1}
   \theta_{m}(\nu_{m}^{i}) := \phi_{m}^{i}(\sigma_m^i(1)), \; i=1,\dots,a_m,
\end{equation}
and
\begin{equation}\label{construction_theta_m_2}
   \theta_{m}(\psi_m^i(j)) := \phi_{m}^{i}(\sigma_m^i(j+1)) \quad \forall j \in \{1,\ldots,c_m^i\}
\end{equation}
for each $i\in\{1,\ldots,a_m\}$ with $c_m^i \neq 0$. On
$\{1,\dots,N\}\setminus\Psi_m$ let $\theta_m$ be the mapping onto
$\{1,\dots,N\}\setminus\Phi_m$ with the same order as $\theta_{m-1}$
restricted to $\{1,\dots,N\}\setminus\Psi_m$, that is, whenever
$\theta_{m-1}(i)<\theta_{m-1}(j)$ for some $i,j\in\{1,\dots,N\}
\setminus\Psi_m$, then $\theta_m(i)<\theta_m(j)$ should also hold.
\smallskip

\begin{example}\label{ex_permutation}

We consider a realisation of the $m$-th event of a population of size $N=8$,
as illustrated in Figure~\ref{fig_reproduction}.
There are $a_m=2$ families (depicted by ``triangle'' and ``star'', respectively).
The first family $\phi_m^1=\{3,6,8\}$ has size $c_m^1+1=3$, the second, $\phi_{m}^2=\{2,5\}$, has size $c_m^2+1=2$.
Hence, the set of levels involved in this birth event is $\Phi_m=\{2,3,5,6,8\}$,
and $\mu_{m}^1=3$, $\mu_m^2=2$ are the levels of the parental particles.
Since there is no parental particle among the highest three levels, the particles
at levels $\Delta_m=\{6,7,8\}$ ``die''.

Now let us assume that $\theta_{m-1}$ is as given in Figure~\ref{theta_m-1}.
Thus, $\nu_m^1=4$, $\nu_m^2=1$, $\psi_m=\{3,5,7\}$.
The set of indices $\psi_m$ of individuals in the Moran model who will get replaced
by offspring in this event is partitioned according to the family sizes,
for example let $\psi_m^1=\{3,7\}$ and $\psi_m^2=\{5\}$.

We construct $\theta_m$ as follows: Let
$\sigma_m^1={ 1\: 2\: 3 \choose 3\: 1\: 2}$
and $\sigma_m^2={ 1\: 2 \choose 2\: 1}$.
For the restriction of $\theta_m$ to $\Psi_m=\{1,3,4,5,7\}$, we read from
\eqref{construction_theta_m_1} that $\theta_m(4)=\phi_m^1(3)=8$, $\theta_m(1)=\phi_m^2(2)=5$
and from \eqref{construction_theta_m_2} that
$\theta_m(3)=\theta_m(\psi_m^1(1))=\phi_m^1(\sigma_m^1(1+1))=\phi_m^1(1)=3$,
$\theta_m(7)=\theta_m(\psi_m^1(2))=\phi_m^1(\sigma_m^1(2+1))=\phi_m^1(2)=6$ and
$\theta_m(5)=\theta_m(\psi_m^2(1))=\phi_m^2(\sigma_m^2(1+1))=\phi_m^2(1)=2$.
This leads to the partial permutation which is given in Figure~\ref{add_families}.

Restricted to the complementary set $\{2,6,8\}$, $\theta_m$ is a
mapping onto $\{1,4,7\}$ with the same order as $\theta_{m-1}$
restricted to $\{2,6,8\}$. The resulting permutation $\theta_m$ is given in
Figure~\ref{theta_m}.\hfill$\blacksquare$

\begin{figure}
\begin{center}
    \setcounter{subfigure}{0}
    \subfigure[Initial permutation $\theta_{m-1}$]{
%       \begin{minipage}{.3\textwidth}
        \centering\includegraphics[scale=1.2]{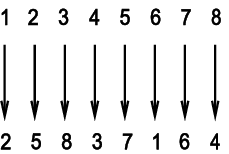}
        \label{theta_m-1}
%       \end{minipage}
    }
    \qquad\qquad
    \subfigure[The families are added]{
%       \begin{minipage}{.3\textwidth}
        \centering\includegraphics[scale=1.2]{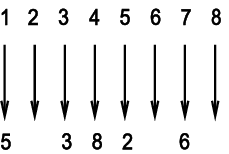}
        \label{add_families}
%       \end{minipage}
    }
    \qquad\qquad
    \subfigure[\ The completed permutation in Example \ref{ex_permutation}]{
%       \begin{minipage}{.3\textwidth}
        \centering\includegraphics[scale=1.2]{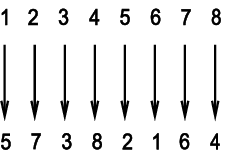}
        \label{theta_m}
%       \end{minipage}
    }
    \caption{The construction of the new permutation from the old permutation carried out in Example~\ref{ex_permutation}}
\end{center}
\end{figure}
\end{example}

For notational convenience, let
\begin{equation}
   \chi_m := (\nu_m^1,\psi_m^1,\ldots,\nu_m^{a_m},\psi_m^{a_m}),
\end{equation}
which summarises the combinatorial information generated in the $m$-th step
(namely, the family structure we would observe in the Moran model).

\begin{lemma}\label{distributed}
For each $m$, $\chi_1,\ldots,\chi_m,\theta_m$ are independent.
Furthermore $\theta_m$ is uniformly distributed over $S_N$ and
\begin{equation}
   \Upsilon_m := \bigcup_{i=1}^{a_m} \{\nu_m^i\} \cup \psi_m^i
\end{equation}
is uniformly distributed over $P_{N,c_m+a_m}$, and each $\chi_m$ is, given $\Upsilon_m$,
uniformly distributed on all ordered partitions of $\Upsilon_m$ with
family sizes consistent with the $c_m^i$.
\end{lemma}

\begin{proof} We prove the statement by induction.
Denoting $\mathcal{F}_m = \sigma(\theta_k, \chi_k: 0 \le k \le m)$, we have
\begin{equation}\label{dependence_through_theta}
   \E[f(\theta_m,\chi_m)\mid \mathcal{F}_{m-1}]=\E[f(\theta_m,\chi_m)\mid \theta_{m-1}],
\end{equation}
since $\theta_m$ and $\chi_m$ are only based on $\theta_{m-1}$ and
additional independent random structure.

This implies, for any choice of $f\colon S_n \to \R$ and
$h_k \colon \cup_{n=1}^N \big(\{1,\dots,N\} \times \mathcal{P}(\{1,\dots,N\})\big)^n \to \R$,
\begin{eqnarray*}
  \E\left[f(\theta_m)\prod_{k=1}^{m}h_k({\chi_k})\right]  & =
  &  \E\left[\E[f(\theta_m)h_m({\chi_m})\mid\mathcal{F}_{m-1}]\prod_{k=1}^{m-1}h_k({\chi_k})\right]\\
  & = & \E\left[\E[f(\theta_m)h_m({\chi_m})\mid\theta_{m-1}]\prod_{k=1}^{m-1}h_k({\chi_k})\right]\\
  & = &  \E[f(\theta_m)h_m({\chi_m})]\prod_{k=1}^{m-1}\E[h_k({\chi_k})],
\end{eqnarray*}
where we used \eqref{dependence_through_theta} in the second and the induction
hypothesis in the third equality. It remains to show
that $\theta_m$ and $\chi_m$ are independent and have the correct distributions.

$\theta_{m-1}$ is uniformly distributed by the induction hypothesis and
independent of the distributions of the parental-levels $\mu_m^i$ and
the ``death-levels'' $\Delta_m$ by construction. It is immediate from the
construction that $\Phi_m$ and $\Upsilon_m$ are uniformly distributed over $P_{N,c_m+a_m}$ and
the family structure $\chi_m$ is uniformly distributed among all
admissible configurations.

Furthermore, conditioning on $\chi_m$ and $\Phi_m$, $\theta_m$ is
uniformly distributed over all permutations that map $\Upsilon_m$
onto $\Phi_m$. This follows from the fact that $\Phi_m$ is uniform
on $P_{N,c_m+a_m}$ and that this set is uniformly divided into
the families $\phi_m^i$. Since uniform and independent
permutations $\sigma_m^i$ are used for the construction of
$\theta_m$ and the non-participating levels remain uniformly distributed,
$\theta_m$ is uniform under these conditions.

Finally, conditioning on $\chi_m$ does not alter the fact that $\Phi_m$ is
uniformly distributed over
$P_{N,c_m+a_m}$. This implies that given $\chi_m$, $\theta_m$ is also
uniformly distributed over $S_N$. Since
\begin{equation}
   \mathcal{L}(\theta_m\vert\chi_m) = \text{unif}(S_N) =
   \mathcal{L}(\theta_m),
\end{equation}
$\theta_m$ and $\chi_m$ are independent of each other.
\end{proof}

\begin{proof}[Proof of Theorem \ref{exchangeable_fixed_times}]
Suppose a realization $X$ of the $N$-particle $(\Xi,B)$-lookdown-model is given and let
$\{t_m\}$ denote the times at which the birth events occur. The families involved
in the $m$-th birth event are denoted by $\phi^i_m$.
Note that by definition of the lookdown-dynamics, the ``ingredients''
$\Phi_m, c_m, a_m, c_m^i, \mu_m^i, \Delta_m$ introduced earlier can be obtained from
this, and that their joint distributions is as discussed above.

Moreover, let the initial permutation $\theta_0$ be independent of
$X$ and uniformly distributed on $S_N$. Let $\sigma^i_m$ be independent of all other
random variables and uniformly distributed on $S_{c^i_m+1}$, $1\le i \le a_m$, $m\in\N$.

Define $\theta_m$ as above, and
\begin{equation}
   \Theta(t) := \theta_m \quad\text{for}\:t_m \le t < t_{m+1}.
\end{equation}
Observe that, by Lemma~\ref{distributed},
\begin{equation}\label{define_y_through_theta}
   (Y_1(t),\ldots,Y_N(t)) := (X_{\Theta_1(t)}(t),\ldots,X_{\Theta_N(t)}(t))
\end{equation}
is a version of the $(\Xi,B)$-Moran-model. Note that ``one-member families''
are in this construction simply treated as non-participating individuals in the
$(\Xi,B)$-Moran-model.

$Y(t)$ depends only on $Y(0)$, $\{\chi_m\}_{t_m \le t}$ and
the the evolution of the type processes between birth and death
events, so $\Theta(t)$, and hence $\Theta(t)^{-1}$ is independent of
\begin{equation}\label{eq_xi_filtration}
   \mathcal{G}_t :=
   \sigma\big((Y_1(s),\ldots,Y_N(s)):s\le t\big) \vee \sigma(\chi_m : m \in \N)
\end{equation}
due to Lemma~\ref{distributed}.
Therefore, we see from
\begin{equation}
   (X_1(t),\ldots,X_N(t)) =
   (Y_{\Theta^{-1}_1(t)}(t),\ldots,Y_{\Theta^{-1}_N(t)}(t))
\end{equation}
that $(X_1(t),\ldots,X_N(t))$ is exchangeable.
\end{proof}

\begin{corollary}
Starting from the same exchangeable initial condition, the laws of the
empirical processes of the $(\Xi,B)$-Moran-model and the
$(\Xi,B)$-lookdown-model coincide.
\end{corollary}

The exchangeability property does not only hold for fixed times, but also
for stopping times.

\begin{theorem}\label{exchangeable_stopping_times}
  Suppose that the initial population vectors $Y^N(0)$ in the
  $(\Xi,B)$-Moran-model and $X^N(0)$ in the $(\Xi,B)$-lookdown-model
  have the same exchangeable distribution, and let $\tau$ be a
  stopping time with respect to $(\mathcal{G}_t)_{t\ge 0}$ given
  by \eqref{eq_xi_filtration}. Then, $(X^N_1(\tau),\dots,X^N_N(\tau))$
  is exchangeable.
\end{theorem}

\begin{proof}
We show that $\Theta(\tau)$ is
independent of the $\sigma$-algebra $\mathcal{G}_\tau$ (the $\tau$-past)
and uniformly distributed over $S_N$.

First, assume that $\tau$ takes only countable many values $t_k$, $k\in\N$.
Let $A\in\mathcal{G}_\tau$ and $h:S_N\to\mathbb{R}_+$, then
\begin{equation}\label{discrete_stopping_times}
\begin{split}
  \mathbb{E}\Big( h\big(\Theta(\tau)\big)\1_A\Big)
  & = \mathbb{E} \Big( \sum_{k=1}^\infty h\big(\Theta(t_k)\big) \1_{A\cap\{ \tau = t_k \}} \Big)\\
%   & \overset{\text{mon. conv.}}= \sum_{k=1}^\infty \mathbb{E} h\big(\Theta(t_k)\big) \1_{A\cap\{ \tau^d = t_k \}}\\
  & = \sum_{k=1}^\infty \Big(\mathbb{E} h\big(\Theta(t_k)\big)\Big) \Big(\mathbb{E}\1_{A\cap\{\tau = t_k\}}\Big)\\
  & = \int h(\Theta)\,\mathfrak{U}(d\Theta)
      \sum_{k=1}^\infty\mathbb{E}\1_{A\cap\{\tau = t_k\}}\\
  & = \int h(\Theta)\,\mathfrak{U}(d\Theta)\,\mathbb{E}\1_A,
\end{split}
\end{equation}
where $\mathfrak{U}$ denotes the uniform distribution on $S_N$.
To see that the second equality holds, observe that, for fixed
$t_k$, $\Theta(t_k)$ is independent of $\mathcal{G}_{t_k}$ in
the proof of Theorem \ref{exchangeable_fixed_times}.

By approximating an arbitrary stopping time from above by a sequence
of discrete stopping times, we see that \eqref{discrete_stopping_times}
holds in the general case as well. Now, exchangeability of
$(X^N_1(\tau),\dots,X^N_N(\tau))$ follows as in the proof
of Theorem~\ref{exchangeable_fixed_times}.
\end{proof}

\begin{remark} \rm
One can also define a variant of the $(\Xi,B)$-lookdown
model which is more in the spirit of the `classical' lookdown construction
from \cite{DK96}, where, instead of a)--c) on page~\pageref{items:atoc},
at a jump time each particle simply copies the type of that member of the family it
belongs to with the lowest level (and no types get shifted upwards). This variant,
which is (up to a renaming of levels by the points of a Poisson process on $\R$)
also the one suggested by adapting \cite{KR08} to the `simultaneous multiple merger'-scenario,
has been considered by Taylor \& V\'eber (2008, personal communication).
The same results as above hold for this variant, with only minor modifications of the
proofs. Note that the flavour of the lookdown process described above is easily
adaptable to a set-up with time-varying total population size, which is not obvious
for the other variant.
\end{remark}

\subsection{The limiting population}
\label{limitpop}

We now construct the limiting $E^\infty$-valued particle system $X = (X_1,X_2,\ldots)$ by formulating a stochastic differential
equation for each level $l$. These exist for each level and are well
defined, since the equation for level $l$ needs only information about
lower levels.

The generator \eqref{mutop} of a pure jump process can be written in the form
\[
    Bf(x) = r\int_0^1 \big( f(m(x,u)) - f(x)\big)\,du,
\]
where $r$ is the global mutation rate and $m\colon E\times[0,1] \to E$
transforms a uniformly distributed random variable on $[0,1]$ into the
jump distribution $q(x,dy)$ of the process. The random times and
uniform ``coins'' for the mutation process at each level $l$ are given
by a Poisson point process $\mathfrak{N}^{\text{Mut}}_l$ on
$\R_+\times[0,1]$ with intensity measure $r dt \otimes du$.

As in Section~\ref{canon_moran_modell}, denote by
\[
    (t_m,{\bzeta}_m, {\bf u}_m) = (t_m, (\zeta_{m1}, \zeta_{m2}, \ldots), (u_{m1}, u_{m2}, \ldots))
\]
the points of the Poisson point process $\mathfrak{M}^{\Xi_{0}}$ and
recall the definition \eqref{eq_definition_g} of the ``colour''
function $g$. Based on this, define
\begin{equation}\label{eq_llj}
   L_J^l(t)\ := \sum_{m:t_m\le t}
   \prod_{j\in J}\1_{\{g({\bzeta}_m, u_{mj})< \infty\}}
   \prod_{j\in\{1,\ldots,l\}\setminus J}\1_{\{g({\bzeta}_m,u_{mj})=\infty\}},
\end{equation}
for $J \subset \{1, \dots ,l\}$ with $|J|\ge 2$\,.
$L^l_J(t)$ counts how many times, among the levels in $\{1,\dots,l\}$,
exactly those in $J$ were involved in a birth event up to time $t$.  Moreover, let
\begin{equation}\label{eq_lljk}
   L_{J, k}^l(t)\ := \sum_{m:t_m\le t}
   \prod_{j\in J}\1_{\{g({\bzeta}_m, u_{mj})=k\}}
   \prod_{j\in\{1,\ldots,l\}\setminus J}\1_{\{g({\bzeta}_m,u_{mj})\neq k\}}.
\end{equation}
$L^l_{J,k}(t)$ counts how many times, among the levels in
$\{1,\dots,l\}$, exactly those in $J$ were involved in a birth event
up to time $t$ and additionally assumed ``colour'' $k$.

To specify the new levels of the individuals not participating in a
certain birth event, we construct a function $J_m$ as follows:

Denote by $\mu_m^k := \min\{l \in \N \,\vert\, g({\bzeta}_m,u_{ml}) =
k\}$ the level of the parental particle of family number $k$ and by
$M_m := \{\mu_m^k\}_{k \in \N}$ the set of all levels of parental
particles involved in the $m$-th birth event. Furthermore $U_m := \{ l
\in \N \,\vert\, g({\bzeta}_m,u_{ml}) = \infty\}$ denotes the set of
the levels not participating in the birth event $m$.
Define the mapping
\begin{equation}\label{eq_function_jm}
   J_m\,:\,U_m\to\N\setminus M_m
\end{equation}
that maps the $i$-th smallest element of the set $U_m$ to the $i$-th
smallest element of the set $\N\setminus M_m$ for all $i$.

Assuming for the moment that $E$ is an Abelian group,
the (infinite) vector describing the types in the $(\Xi,B)$-lookdown-model
is defined as the (unique) strong solution of the following system of
stochastic differential equations.
The lowest individual on level 1 just evolves according to mutation, i.e.,
\begin{equation} \label{eq_stochastic_differential_equation_lvel_one}
    X_1(t)\ :=\ \int_{[0,t]\times[0,1]}
    (m(X_1(s-),u) - X_1(s-))\,d\mathfrak{N}^{\text{Mut}}_1(s,u).
\end{equation}
The individuals above level one can look down during birth events. Thus,
for $l\ge 2$, define
\begin{equation}\label{eq_stochastic_differential_equation}
\begin{split}
    X_l(t) := & X_l(0) + \int_{[0,t]\times[0,1]} \big(m(X_l(s-),u) - X_l(s-)\big)\,d\mathfrak{N}^{\text{Mut}}_l(s,u)\\
        & + \sum_{1 \leq i <l} \int_0^t (X_i(s-) - X_l(s-))\,d\mathfrak{N}^K_{il}(s)\\
        & + \sum_{1 \leq i < j < l} \int_0^t (X_{l-1}(s-) - X_l(s-))\,d\mathfrak{N}^K_{ij}(s)\\
        & + \sum_{k \in \N} \sum_{K \subset \{1,\ldots,l\},l \in K} \int_0^t (X_{\min(K)}(s-) - X_l(s-))\,dL^l_{K,k}(s)\\
        & + \sum_{K \subset \{1,\ldots,l\},l \notin K} \int_0^t (X_{J_m(l)}(s-) - X_l(s-))\,dL^l_K(s).\\
\end{split}
\end{equation}

The second and third lines describe the ``Kingman events'', where only
pairs of individuals are involved. The first part copies the type from
level $i$ when $l$ looks down to this level, because it is involved in
a birth event and the parental particle is at level $i$. The second
part handles the event that the parental particle places a child on a
level below $l$. In this case, $l$ has to copy the type from the level
$l-1$, since the new individual is inserted at some level below $l$
and pushes all particles above that level one level up.

The fourth and fifth lines describe the change of types for a birth
event with large families in a similar way. If the particle at level
$l$ is involved in the family $k$, it copies the type from the
parental particle which resides at the lowest level of the family. If
level $l$ is not involved in any family, then $J_m(l)$ ($\le l$)
gives the level from where the type is copied (which comes from shifting
particles not involved in the lookdown event upwards).

Since the equation for $X_l$ involves only $X_1,\dots,X_l$ and
finitely many Poisson processes, it is immediate that there exists
a unique strong solution of
\eqref{eq_stochastic_differential_equation_lvel_one}--\eqref{eq_stochastic_differential_equation}.

In the case where $E$ has no group structure, one may still construct
suitable jump-hold processes $X_i$, using the driving Poisson processes in an
obvious extension of \eqref{eq_stochastic_differential_equation_lvel_one}--\eqref{eq_stochastic_differential_equation}.

These stochastic differential equations determine an infinitely large
population vector
\begin{equation}
    X(t)\ :=\ (X_1(t),X_2(t),\ldots)
\end{equation}
in a consistent way, and for each $N\in\N$, the dynamics of $(X_1,\ldots,X_N)$
is identical to that defined in Section~\ref{orderedmodel}. In particular,
we see from Theorem~\ref{exchangeable_fixed_times} that, for each $t\ge 0$,
$X(t)$ is exchangeable and the empirical distribution
\begin{equation}\label{eq_emp_dist_fixed_time_exists}
   Z(t)
   \ :=\ \lim_{l\to\infty} Z^l(t)
   \ :=\ \lim_{l\to\infty}\frac{1}{l}
         \sum_{i=1}^l \delta_{X_i(t)}
\end{equation}
exists almost surely. Let $F$ be the set of bounded measurable functions
$\varphi:[0,\infty)\times [0,1]^\N\times [0,1]^\infty\to\R$ such that
$\varphi(t,\bzeta,\bf{u})$ does not depend on $\bf{u}$, and put
\begin{equation} \label{eq_def_Ht}
   \mathcal{H}_t\ :=\ \sigma\bigg(
      \big(Z(s):s\le t\big),
      \Big({\textstyle\int\varphi\,d\mathfrak{M}^{\Xi_0}}:\varphi\in F\Big)
   \bigg).
\end{equation}
\begin{corollary}
Let $\tau$ be a stopping time with respect to $(\mathcal{H}_t)_{t\ge 0}$.
Then
\begin{equation}
    X(\tau)\ =\ (X_1(\tau),X_2(\tau),\ldots)
\end{equation}
is exchangeable.
\end{corollary}

\begin{proof}
We claim that for $t\ge 0$, $A\in\mathcal{H}_t$ with $\P\{A\}>0$ and $n\in\N$,
\begin{equation}
\label{eq_exch_conditioned}
(X_1(t),\dots,X_n(t))\mbox{ is exchangeable under $\P\{\cdot|A\}$}.
\end{equation}
Observe that, taking $A=\{\tau=t_k\}$, \eqref{eq_exch_conditioned}
immediately implies the result for discrete stopping times $\tau$, from
which the general case can be deduced by approximation as in the proof
of Theorem~\ref{exchangeable_stopping_times}.

Obviously, \eqref{eq_exch_conditioned} is equivalent to
\begin{equation}
\label{eq_exch_conditioned2}
\P\big\{A \cap \{ (X_1(t),\dots,X_n(t)) \in C \} \big\}
= \P\big\{A \cap \{ (X_{\sigma(1)}(t),\dots,X_{\sigma(n)}(t)) \in C \} \big\}
\quad \forall\, C  \subset E^n, \sigma \in S_n.
\end{equation}
%for any $C \subset E^n$ and $\sigma \in S_n$.
As the collection of sets $A$ from $\mathcal{H}_t$ satisfying
\eqref{eq_exch_conditioned2} is a Dynkin system, it suffices to verify
\eqref{eq_exch_conditioned2} for events of the form
\begin{equation} \label{eq_A_cap_stable}
   A\ =\ \{Z(s_1)\in B_1,\ldots,Z(s_k)\in B_k\}\cap H',
\end{equation}
where $H' \in \sigma\big( \int \varphi \, d\mathfrak{M}^{\Xi_{0}} : \varphi \in F \big)$,
$k\in\N$, $s_1<\cdots<s_k\le t$, %%$B_1,\dots,B_k \subset \mathcal{M}_1(E)$
$B_i\in\mathcal{B}(s_i)$ for $i\in\{1,\ldots,k\}$, and $\mathcal{B}(s_i)$ is a $\cap$-stable
generator of $\mathcal{B}_{\mathcal{M}_1(E)}$ with the property that
$\P\{Z(s_i) \in \partial B'\}=0$ for all $B'\in \mathcal{B}(s_i)$.

For $A$ as given in \eqref{eq_A_cap_stable}, $\varepsilon > 0$
and $n \in\N$, $\sigma\in S_n$, $C\subset E^n$ appearing in \eqref{eq_exch_conditioned2},
by \eqref{eq_emp_dist_fixed_time_exists} there exists $l$ ($l \gg n$) such that
\[
A_l\ :=\ \{Z^l(s_1)\in B_1,\ldots,Z^l(s_k)\in B_k\}\cap H'
\]
satisfies $\P\{(A\setminus A_l)\cup (A_l\setminus A)\}\le\varepsilon$.
By the arguments given in the proof of Theorem~\ref{exchangeable_stopping_times},
\eqref{eq_exch_conditioned2} holds with $A$ replaced by $A_l$.
Finally, take $\varepsilon\to 0$ to conclude.
\end{proof}

\section{Pathwise convergence: Proof of Theorem~\ref{main}}
\label{tightness}

Recall the empirical processes $Z^l$, and their limit $Z$, from
\eqref{eq_emp_dist_fixed_time_exists}. Obviously, for each $l\in\N$,
the process $(Z^l(t))_{t\ge 0}$ has c\`adl\`ag paths. To
verify the corresponding property for $Z$, we introduce the following
auxiliary (L\'evy) process $U$, derived from Poisson point process
$\mathfrak{M}^{\Xi_{0}}$ which governs the large family birth events
of the population $X$: If $\big\{(t_m,{\bzeta}_m,{\bf u}_m)\big\}$
are the points of the process $\mathfrak{M}^{\Xi_0}$, we define
\begin{equation}
    U(t)\ :=\ \sum_{t_m\le t}v_m^2,
\end{equation}
where $v_m:=\sum_{i=1}^\infty\zeta_{mi}$. The jumps of $U:=(U(t))_{t\ge 0}$ are
the squared total fractions of the population which are replaced in
large birth events. The generator of $U$ is given by
\begin{equation}
    Df(u)\ =\ \int_0^1(f(u+v^2)-f(u))\,\nu(dv),
\end{equation}
where the measure $\nu$ on $[0,1]$, defined via
\begin{equation}
    \nu(A)\ :=\ \int_\Delta \1_{\{\sum_{i=1}^\infty\zeta_i\in A\}}
    \frac{\Xi(d\bzeta)}{({\bzeta},{\bzeta})},
\end{equation}
governs the jumps.

We need the following version of Lemma~A.2 from \cite{DK99}.
\begin{lemma}\label{A2} a)
Let $e_1,e_2,\dots$ be exchangeable and suppose there exists a
constant $K$ such that $|e_i|\le K$ almost surely. Define
\begin{equation}
    M_k\ :=\ \frac{1}{k}\sum_{i=1}^k e_i
\end{equation}
and let $M_\infty$ be the almost sure limit of $(M_k)_{k\in\N}$, whose
existence is guaranteed by the de Finetti Theorem.  Let $\varepsilon>0$.
Then there exists $\eta_1>0$ depending only on $K$ and $\varepsilon$,
such that, for
$l<n \in \N \cup \{\infty\}$,
\begin{equation}
    \P\{\left|M_n-M_l\right|\ge \varepsilon\} \le
    2e^{-\eta_1(K,\varepsilon) l}.
\end{equation}

b) Let $(e_i(t))_{t\in[0,1]}$ be centered martingales such that
$\max_{i \in\N} \sup_{t \in [0,1]} |e_i(t)| \le K$ almost surely and
$(e_1(1),e_2(1),\ldots)$ is exchangeable. Put
\[
M_k(t)\ :=\ \frac{1}{k}\sum_{i=1}^k e_i(t).
\]
Let $\varepsilon>0$. Then there
exists $\eta_2>0$ depending only on $K$ and $\varepsilon$,
such that, for $l\in\N$ % $l<n \in \N \cup \{\infty\}$,
\begin{equation}
    \P\{\sup_{t\in[0,1]} |M_k(t)| \ge \varepsilon\} \le
    2e^{-\eta_2(K,\varepsilon) l}.
\end{equation}
\begin{proof}
The proof of part~a) is a straightforward extension of
that of Lemma A.2 from \cite{DK99}, which employs the fact that an
infinite exchangeable sequence is conditionally i.i.d.~together with
standard arguments based on the moment generating function.

For part~b) observe that by Doob's submartingale inequality,
\begin{equation} \label{doob_inequality}
   \P\Big\{\sup_{0\le t<1} |M_{k}(t)|\ge\varepsilon\Big\}
    \ \le\ \inf_{\lambda>0} \frac{1}{e^{\varepsilon\lambda}}
            \E e^{\lambda |M_{k}(1)|}
    \ \le\ \inf_{\lambda>0} \frac{1}{e^{\varepsilon\lambda}}
            \E\exp\Big(\frac{\lambda}{k}\sum_{i=1}^k |e_i(1)|\Big).
% \end{split}
\end{equation}
Now proceed as in part~a).
\end{proof}
\end{lemma}

The following lemma provides the technical core of the argument
and replaces Lemma~3.4 and Lemma~3.5 in \cite{DK99}. The proof given below
follows closely the arguments of Donnelly and Kurtz \cite{DK99}.
\begin{lemma}\label{lemma_for_bound}
  In the setting of Theorem~\ref{main}, for all $c,T,\epsilon
  > 0$ and $f \in \mathcal{D}(B)$ (the domain of the mutation generator) there exists a sequence
  $\delta_l$ such that $\sum_{l=1}^\infty\delta_l < \infty$ and
\begin{equation}
  \P \Big\{ \sup_{0\le t\le T}\big\vert\langle f,Z(t) \rangle - \langle f,Z^l(t) \rangle\big\vert \ge 11\epsilon, U(T)\le c \Big\} \le \delta_l.
\end{equation}
\end{lemma}

\begin{proof}
By Lemma~\ref{A2} and the exchangeability properties of $X$, we have
\begin{equation}
\label{eq:fstopempir}
  \P\{|\langle f,Z(\alpha)\rangle - \langle f,Z^l(\alpha)\rangle|
  \ge\epsilon\}\ \le\ 2e^{-\eta l},
\end{equation}
if $\alpha$ is a stopping time with respect to
$\widetilde{\mathcal{H}}:=(\widetilde{\mathcal{H}}_t)_{t\ge 0} := \big(\sigma(U(s):s\ge0) \vee
\sigma(Z(s):0\le s\le t)\big)_{t\ge 0}$
(observe that $\widetilde{\mathcal{H}}_t \subset \mathcal{H}_t$, where
$\mathcal{H}_t$ is defined in \eqref{eq_def_Ht}).

Now fix $l$ and $\epsilon$. Define the $\widetilde{\mathcal{H}}$-stopping times
\begin{equation}
    \alpha_1\ :=\ \inf \left\{t:
    U(t)>\frac{1}{l^4}\right\}\wedge\frac{1}{l^4}
\end{equation}
and
\begin{equation}
    \alpha_{o+1}\ :=\ \inf\left\{t: U(t)>U(\alpha_o)+\frac{1}{l^4}
    \right\}\wedge\left(\alpha_o+\frac{1}{l^4}\right), \quad o=1,2,\dots,
\end{equation}
which yield a decomposition of the interval $[0,T]$.
Note that on the event $\big\{ U(T) \le c\big\}$ there exist at most
\begin{equation}\label{definition_o_l}
    o_l:=2(c+T)l^4
\end{equation}
such $\alpha_o$, i.e., we have
\begin{equation}
    \P\big\{\alpha_{o_l} < T, U(\alpha_{o_l}) < c, U(T) \le c\big\}=0.
\end{equation}

We define a second kind of $\widetilde{\mathcal{H}}$-stopping
times depending on $\alpha_k$ via
\begin{equation}\label{definition_alpha_o}
    \tilde{\alpha}_o := \inf\{t>\alpha_o :\vert \langle f,Z(t)\rangle - \langle f,Z(\alpha_o)\rangle \vert \ge 6\epsilon\}.
\end{equation}

We see from (\ref{eq:fstopempir}) that
\begin{equation}
  H_o:=\vert\langle f,Z(\alpha_o)\rangle - \langle f,Z^l(\alpha_o)\rangle\vert \vee \vert\langle f,Z(\tilde{\alpha}_o)\rangle - \langle f,Z^l(\tilde{\alpha}_o)\rangle\vert
\end{equation}
satisfies
\begin{equation}\label{bounded_h_supremum}
    \P\Big\{\sup_{o\le o_l}H_o\ge\varepsilon, U(T) \le c  \Big\}\le
    \sum_{o=1}^{o_l} \P\left\{H_o\ge\varepsilon,  U(T) \le c \right\}\le
    8(c+T)l^4e^{-\eta l}.
\end{equation}

It remains to estimate the variation of $Z^l$ and $Z$ in between the
stopping times $\alpha_o$.
For $u\in [\alpha_o,\alpha_{o+1})$ let $\beta_{jo}(u)$ denote the smallest index of a descendant of
$X_j(\alpha_o)$, let the stopping time $\gamma_{jo}$ be the time
when the smallest descendant of $X_j(\alpha_o)$ is shifted
above the level $l$. Put
\[
\tilde{X}_j(u)=
\begin{cases} X_{\beta_{jo}(u)}(u) & \mbox{if} \;  u < \gamma_{jo},\\
  X_{\beta_{jo}(\gamma_{jo}-)}(\gamma_{jo}-) & \mbox{if} \;  u \ge \gamma_{jo}.
\end{cases}
\]

Observe that
\begin{equation}
\begin{split}
  \langle f,Z^l(u)\rangle &- \langle f,Z^l(\alpha_o)\rangle=
  \langle f,Z^l(u)\rangle-\frac{1}{l} \sum_{j=1}^l
  f(\tilde{X}_j(u))+\frac{1}{l} \sum_{j=1}^l \Big(f(\tilde{X}_j(u)) -
  f(\tilde{X}_j(\alpha_o)) \Big).
\end{split}
\end{equation}

It will be useful to treat the two parts of the sum separately. Define
\begin{equation*}
    K_1:=\max_{o\le o_l}\sup_{u\in[\alpha_o,\alpha_{o+1})}\bigg|\langle f, Z^l(u)\rangle - \frac 1 l \sum_{j=1}^l f(\tilde{X}_j(u))\bigg|
\end{equation*}
and
\begin{equation*}
    K_2:=\max_{o\le o_l}\sup_{u\in[\alpha_o, \alpha_{o+1})}\bigg|\frac 1 l \sum_{j=1}^l \big(f(\tilde{X}_j(u))-f(\tilde{X}_j(\alpha_o))\big)\bigg|.
\end{equation*}
Note that the law of $K_2$ depends only on the mutation mechanism,
since $\tilde{X}_j(u)$ follows the line of the individual
$\tilde{X}_j(\alpha_o) = X_j(\alpha_o)$ and thus only evolves
independently according to a mutation process with generator $B$.

Begin with $K_1$ and note that, for $u\in[\alpha_o,\alpha_{o+1})$,
\begin{equation}\label{part_of_k_bounded}
   \langle f,Z^l(u)\rangle-\frac{1}{l}\sum_{j=1}^l f(\tilde{X}_j(u))
   = \frac{1}{l}\Big(\sum_{j=1}^l f(X_j(u))-\sum_{j=1}^l f(\tilde{X}_j(u))\Big)
   \le \frac{2\|f\|}{l} N^l[\alpha_o,\alpha_{o+1}),
\end{equation}
where $N^l[\alpha_o,\alpha_{o+1})$ is the total number of
births occurring in the time interval $[\alpha_o,\alpha_{o+1})$ with
index less than or equal to $l$. To see this note that at time
$\alpha_o$ the two sums in the second expression cancel. A birth event
in the interval $[\alpha_o,\alpha_{o+1})$ means that one type is
removed from the second sum and another one is added, thus the
expression can be altered by up to $2||f||/l$.

There are two mechanisms which can increase
$N^l[\alpha_o,\alpha_{o+1})$. It can either increase during a large birth event given by a ``jump''
of $\mathfrak{M}^{\Xi_{0}}$ or during a small birth event which is given by one of the
``Kingman-related'' Poisson-Processes $\mathfrak{N}^K_{ij}$.

We first consider large birth events. Let $(v_i)$ be the jumps
of $U$ in the interval $[\alpha_o, \alpha_{o+1})$, and condition on
this configuration for the rest of this paragraph. At the time of the
$m$-th jump, a Binomial($l, v_m$)-distributed number of levels $\le l$
participates in this event, hence $k_m$, the total number of children
below level $l$ in the $m$-th birth event, satisfies
\[
k_m \le (b_m-1)_+,
\]
where $b_m$ is Binomial($l, v_m$)-distributed.
Note that we can subtract 1 from the binomial random variable,
since at least one of the levels participating in the birth event
must be a mother. This subtraction will be crucial later on.

By elementary calculations with Binomial distributions, involving
fourth moments, similar to \cite[p.~186]{DK99}, we can estimate
\begin{equation}
\label{eq:sprich}
 \P\Big\{ \sum_m k_m > \epsilon l \Big\}
\le \P\Big\{ \sum_m (b_m-1)_+ > \epsilon l \Big\} \le \frac{C_1}{l^6}
\end{equation}
for some $0 < C_1 < \infty$. %%Details can be found in the appendix.

As we mentioned before, $N^l[\alpha_o,\alpha_{o+1})$ and thus $K_1$
can also be increased by the Kingman part of the birth process, but
only if the parental particle and its offspring %%mother and the child
are placed below level $l$. The
number of times this happens in the interval $[\alpha_o,\alpha_{o+1})$
is stochastically dominated by a Poisson distributed random variable $R$
with parameter ${l \choose 2}l^{-4}$ since the length of the interval
is bounded by $l^{-4}$. So, the probability that
$\frac{2\|f\|}{l}N^l[\alpha_o,\alpha_{o+1})$ exceeds $2\epsilon$ due
to this mechanism is bounded by the probability that $R$ exceeds
$\frac{l\epsilon}{\|f\|}$.
By elementary estimates on the tails of Poisson random variables, we have
\begin{equation}
  \label{eq:Poisstail}
  \P\Big\{R > \frac{l\epsilon}{\|f\|}\Big\} \le e^{-\eta_1 l},
\end{equation}
for some $\kappa>0$ and $l$ large enough.

Combining (\ref{eq:sprich}) and (\ref{eq:Poisstail}), we obtain
\begin{equation}\label{k1_final_bound}
\begin{aligned}
\P\big\{ K_1 > 2\epsilon, U(T) \le c \big\}
& = \P\Big\{ \max_{o\le o_l}\sup_{u\in[\alpha_o,\alpha_{o+1})}\big\vert\langle f, Z^l(u)\rangle - \frac 1 l \sum_{j=1}^l f(\tilde{X}_j(u)) \big\vert > 2\epsilon, U(T) \le c \Big\}\\
& \le o_l \Big( \frac{C_1}{l^{6}} + e^{-\eta_1 l} \Big),
\end{aligned}
\end{equation}
for $l$ large enough. This controls the increments of $\langle
f,Z^l\rangle$ in the intervals $[\alpha_o,\alpha_{o+1})$.

We now consider $K_2$. Observe that
\begin{align}
  \frac 1 l \sum_{j=1}^l (f(\tilde{X}_j(u))-f(\tilde{X}_j(\alpha_o)))
  & = \frac{1}{l} \sum_{j=1}^l \bigg(f(\tilde{X}_j(u))-f(\tilde{X}_j(\alpha_o)) - \int^u_{\alpha_o} Bf(\tilde{X}_j(s))ds \bigg) \notag \\
  & \quad + \frac{1}{l} \sum_{j=1}^l \int^u_{\alpha_o}
  Bf(\tilde{X}_j(s))ds,
\end{align}
and that, for $u \ge \alpha_o$ and each $o$,
\begin{equation}
    M_{lo}(u \wedge \alpha_{o+1}) := \frac{1}{l} \sum_{j=1}^l \bigg(f(\tilde{X}_j(u\wedge \alpha_{o+1}))-f(\tilde{X}_j(\alpha_o))
                                            - \int^{u\wedge \alpha_{o+1}}_{\alpha_o} Bf(\tilde{X}_j(s))ds \bigg)
\end{equation}
is a martingale.
For $l$ so large that $l^{-4}\|Bf\| \le \varepsilon$, we have
\begin{equation}
\begin{split}\label{k2_first_bound}
    \P\big\{K_2 \ge 2\varepsilon, U(T) \le c \big\} & \le \sum^{o_l-1}_{o=0} \P\Big\{\sup_{\alpha_o \le u < \alpha_{o+1}} \vert M_{lo}(u)
                  + \frac{1}{l} \sum_{j=1}^l \int^u_{\alpha_o} Bf(\tilde{X}_j(s))ds \vert \ge 2\varepsilon, U(T) \le c\Big\}\\
        & \le \sum^{o_l-1}_{o=0} \P\Big\{\sup_{\alpha_o \le u < \alpha_{o+1}} \vert M_{lo}(u) \vert + l^{-4}\|Bf\| \ge 2\varepsilon, U(T) \le c\Big\}\\
        & \le \sum^{o_l-1}_{o=0} \P\Big\{\sup_{\alpha_o \le u < \alpha_{o+1}} \vert M_{lo}(u) \vert \ge \varepsilon, U(T) \le c \Big\}.
\end{split}
\end{equation}
We now need to bound each summand. Using the notation
\[
M_{lo}(u) = \frac{1}{l}\sum_{j=1}^l e_j(u),
\]
where
\begin{equation}
e_j(u) := f(\tilde{X}_j(\alpha_{o+1} \wedge
u))-f(\tilde{X}_j(\alpha_o)) - \int^{\alpha_{o+1}\wedge u}_{\alpha_o}
Bf(\tilde{X}_j(s))ds, \quad u \in [0,1],
\end{equation}
each $(e_i(u))_u$ is a martingale with $\E e_j(u)=0$ and $|e_j(u)|\le
2\|f\|+\|Bf\|/l^4=:K$ almost surely. Moreover, the $e_j(u)$ are
exchangeable. We obtain from Lemma~\ref{A2}
\begin{equation}
\label{condlargedev}
\P\Big\{\sup_{\alpha_o \le u < \alpha_{o+1}} \vert M_{lo}(u) \vert \ge \varepsilon\Big\}
             \le 2 e^{-\eta_2 l},
\end{equation}
for some $\eta_2>0$.

Combining this result with \eqref{k2_first_bound}, we arrive at
\begin{equation}\label{k2_final_bound}
   \P \big\{ K_2 \ge 2\varepsilon, U(T) \le c \big\} \le o_l C_2 e^{-\eta_2 l}.
\end{equation}

Now observe that if $\max_{o\le o_l} H_o  <\epsilon$,
$K_1<2\epsilon$ and $K_2<2\epsilon$, then $\tilde{\alpha}_o\ge
\alpha_{o+1}$. This can easily be seen by contradiction. Indeed,
if we assume that
$\tilde{\alpha}_o < \alpha_{o+1}$, this would imply
\begin{equation}\label{restating_definition}
    \vert \langle f,Z(\alpha_o)\rangle - \langle f,Z(\tilde{\alpha}_o)\rangle \vert \ge 6\epsilon,
\end{equation}
according to \eqref{definition_alpha_o}.
But on the other hand we know that
\begin{equation}
  \vert\langle f,Z(\alpha_o)\rangle - \langle f,Z^l(\alpha_o)\rangle\vert < \epsilon \text{ and }\vert\langle f,Z(\tilde{\alpha}_o)\rangle
   - \langle f,Z^l(\tilde{\alpha}_o)\rangle\vert < \epsilon \quad \forall o
\end{equation}
due to our bound on $H_o$. Since the distance between $\langle
f,Z\rangle$ and $\langle f,Z^l\rangle$ was at most $\epsilon$ at the
beginning of the interval and $\langle f,Z^l\rangle$ can only have
moved by at most $4\epsilon$ on the event $\{ K_1 \le 2\epsilon
\}\cap\{ K_2 \le 2\epsilon \} \cap \{ \max_{o\le o_l}H_o \le
\epsilon\}$,
\begin{equation}
    \vert \langle f,Z(\alpha_o)\rangle - \langle f,Z^l(\tilde{\alpha}_o)\rangle \vert < 5\epsilon
\end{equation}
must hold if $\tilde{\alpha}_o \le \alpha_{o+1}$. But equation \eqref{restating_definition} states that, $\langle
f,Z(\tilde{\alpha}_o)\rangle$ is more than $6\epsilon$ away from its
starting point, so this contradicts that it can only be $\epsilon$
away from $\langle f,Z^l(\tilde{\alpha}_o)\rangle$ which is ensured by
our condition on $H_o$.  Thus $\tilde{\alpha}_o$ has to be greater
than $\alpha_{o+1}$ which in turn implies that
\begin{equation}\label{eq_sup_for_z}
    \sup_{\alpha_o \le u < \alpha_{o+1}}\Big\{\big\vert \langle f,Z(u)\rangle - \langle f,Z(\alpha_o)\rangle \big\vert\Big\} \le 6\epsilon
\end{equation}
holds on the event $\{ K_1 \le 2\epsilon \}\cap\{ K_2 \le 2\epsilon \} \cap \{ \max_{o\le o_l}H_o \le \epsilon\}$.

Putting observations \eqref{bounded_h_supremum} and \eqref{eq_sup_for_z},
the bound \eqref{k2_final_bound} and the bound \eqref{k1_final_bound}
together, we finally obtain
\begin{equation}
  \P \Big\{ \sup_{0\le t\le T}\big\vert\langle f,Z(t) \rangle - \langle f,Z^l(t) \rangle\big\vert \ge 11\epsilon, U(T)\le c \Big\} \le \delta_l
\end{equation}
with
\begin{equation}
    \delta_l := 8(c+T)l^4e^{-\eta l} + o_l C_1 l^{-6} + o_l e^{-\eta_1 l} + o_l C_2 e^{-\eta_2 l}
\end{equation}
which is the statement of the lemma since due to equation \eqref{definition_o_l} $o_l\sim l^4$ holds and therefore the $\delta_l$ are summable.
\end{proof}

\begin{proof}[Proof of Theorem~\ref{main}]
  Almost sure convergence of $Z^l$ to $Z$ with respect to the metric
  \eqref{metric_on_paths_of_measures} follows directly from
  Lemma~\ref{lemma_for_bound} and the Borel-Cantelli Lemma,
  completing the proof of Theorem~\ref{main}.
%\todo{some more comment on finishing the complete proof, maybe the metric before the theorem}
\end{proof}

\section{The Hille-Yosida approach}

In this section we provide two alternative representations of the
$\Xi_0$-Fleming-Viot generator, leading to the distributional
duality to the $\Xi$-coalescent discussed in Section~\ref{dualities},
and we show that they generate a Markov semigroup on $\M_1(E)$, hence
leading to a classical construction of the $\Xi_0$-Fleming-Viot
process as a Markov process.

\subsection{Two representations of the $\Xi_0$-Fleming-Viot generator}

Recall that if the type space $E$ is a compact Polish space (which is
assumed in this paper), then the set $\M_1(E)$ of all
probability measures on $E$, equipped with the weak topology, is
again a Polish space. We briefly recall the notation from
Section~\ref{sec:introduction}. For $f:E^n\to\R$ bounded and measurable
consider the test function
\begin{equation} \label{eq:testfunction1}
   G_f(\mu)\ :=\ \int_{E^n} f(x_1,\dots,x_n)\,\mu^{\otimes n}(dx_1,\dots,dx_n),
   \quad\mu\in\M_1(E).
\end{equation}
The linear operator $L^{\Xi_0}$ was defined via
\begin{equation} \label{eq:genXiFV1}
L^{\Xi_{0}} G_f(\mu) = \int_\Delta \int_{E^\N}
               \left[ G_f\big( (1-|\bzeta|)\mu +
                             {\textstyle\sum_{i=1}^\infty \zeta_i \delta_{x_i}} \big)
                      - G_f(\mu) \right] \mu^{\otimes \N}(d\mathbf{x})
             \frac{\Xi_0(d\bzeta)}{(\bzeta,\bzeta)}.
\end{equation}
This operator is the $\Xi_0$-Fleming-Viot generator from Proposition~\ref{mainprop}.
The following representation will be useful to establish the duality with the
$\Xi_0$-coalescent. Note that if $\Xi$ is concentrated on
$\{\bzeta\in\Delta\,:\,\zeta_i=0\mbox{ for all } i\ge 2\}$, i.e., if the
corresponding coalescent is a $\Lambda$-coalescent, then this result
has already been obtained by Bertoin and Le~Gall \cite[Eqs. (16) and (17)]{BLG03}.

For convenience, we will denote the transition rates by
\begin{equation}\label{eq:xi.rate}
    \lambda(k_1,\dots,k_p)\ =\ \lambda_{b;k_1,\dots,k_r;s},
\end{equation}
where $k_1\ge\dots\ge k_r\ge 2$, $p-r=s$ and $k_{r+1}=\ldots=k_p=1$.
Furthermore, define for $p,n_1,\dots,n_p\in\N$ such that
$n_1+\dots+n_p>p$ ($\Leftrightarrow$ not all $n_i=1$)
\begin{equation} \label{eq:xi.rate.alt}
   \lambda(n_1,\dots,n_p)\ :=\ \lambda(k_1,\dots,k_p),
\end{equation}
where $k_1\ge\dots\ge k_p$ is the re-arrangement of $n_1,\dots,n_p$ in
decreasing order.
\begin{lemma}\label{lemma_generator_alt}
   The operator $L^{\Xi_0}$ has the alternative representation
   \begin{equation} \label{eq:genXiFV.alt}
      L^{\Xi_{0}} G_f(\mu)
      \ =\ \sum_{\pi=\{A_1,\dots,A_p\}\in\mathcal{P}_n \atop
      \text{not all singletons}}
      \hspace{-1em}
      \lambda(|A_1|,\dots,|A_p|)\int_{E^n}
      \left(f\big(\mathbf{x}[\pi]\big)-f(\mathbf{x})\right)
      \mu^{\otimes n}(dx_1,\dots,dx_n),
   \end{equation}
   where $\mathbf{x}[\{A_1,\dots,A_p\}]\in E^n$ has entries
   \[
   (\mathbf{x}[\{A_1,\dots,A_p\}])_i\ :=\ x_{\min A_j}
   \quad\text{if\ \ $i\in A_j$, $i=1,\dots,n$.}
   \]
\end{lemma}
\begin{remark}
   Note that \eqref{eq:genXiFV.alt} basically boils down to
   \eqref{eq_lambda_fleming_viot}, if $|A_i|=1$ for all but one $A_i$.
\end{remark}
\begin{proof}[Proof of Lemma~\ref{lemma_generator_alt}]
   First note that for fixed $\bzeta$ and $\mathbf{x}$,
   \begin{equation} \label{eq:funnysum1}
   \begin{aligned}
      G_f\big( &(1-|\bzeta|)\mu +
                   {\textstyle\sum_{i=1}^\infty \zeta_i \delta_{x_i}}\big) \\
      & =  \sum_{\phi : \{1,\dots,n\}\to\Z_+}
      (1-|\bzeta|)^{a(\phi)}
      \prod_{j \le n\, : \phi(j)>0} \zeta_{\phi(j)}
      \int_{E^{a(\phi)}} f\big(\eta(\phi, \mathbf{x}, \mathbf{y})\big)
        \mu^{\otimes a(\phi)}(dy_1,\dots,dy_{a(\phi)}),
   \end{aligned}
   \end{equation}
   where $a(\phi):=\#\{1\le j\le n:\phi(j)=0\}$ and   $\eta(\phi,\mathbf{x},\mathbf{y})\in E^n$ is given by
   \[
   \eta(\phi, \mathbf{x}, \mathbf{y})_j
   \ =\
   \left\{\begin{array}{cl}
      x_{\phi(j)} & \mbox{if}\; \phi(j) > 0, \\[1ex]
      y_k         & \mbox{if}\; \phi(j) = 0, \;
      \mbox{where}\; k = \#\{ 1 \le j' \le j : \phi(j') = 0\}.
   \end{array}\right.
   \]
   Identity (\ref{eq:funnysum1}) can be understood as follows:
   Expanding the $n$-fold product of $(1-|\bzeta|)\mu +
   \sum_{i=1}^\infty \zeta_i \delta_{x_i}$, we put $\phi(j)=0$
   if in the $j$-th factor, we use $(1-|\bzeta|)\mu$, and we
   put $\phi(j)=i$ if we use $\zeta_i \delta_{x_i}$ in the $j$-factor.

   Each $\phi:\{1,\dots,n\}\to\Z_+$ is uniquely described by a
   partition $\pi=\{A_1,\dots,A_p\} \in \mathcal{P}_n$ with labels
   $\ell_1,\dots,\ell_p \in \Z_+$ by defining $j \sim_\phi j'$ if and only if
   $\phi(j)=\phi(j')>0$ and putting $\ell_i := \phi(A_i)$, $i=1,\dots,p$.
   Note that for a given partition $\{A_1,\dots,A_p\}$, any vector
   $(\ell_1,\dots,\ell_p) \in \Z_+^p$ of labels with the properties
   \[
   \ell_i=0 \; \Rightarrow \; |A_i|=1 \quad \mbox{and} \quad
   i \neq j, \ell_i, \ell_j \neq 0 \; \Rightarrow \; \ell_i \neq \ell_j
   \]
   is admissible. Thus we have
   \begin{eqnarray}
   \notag
   \lefteqn{\int_{E^\N} G_f\big( (1-|\bzeta|)\mu +
                   {\textstyle\sum_{i=1}^\infty \zeta_i \delta_{x_i}}\big)
           \mu^{\otimes \N}(d\mathbf{x})} \\
   \label{eq:altform1}
   & = &
   \sum_{\pi=\{A_1,\dots,A_p\} \in \mathcal{P}_n}
   \sum_{(\ell_1,\dots,\ell_p) \atop
     \text{admissible}} \hspace{-1em}
   (1-|\bzeta|)^{\#\{1\le i\le p : \ell_i=0\}}
   \prod_{i=1 \atop \ell_i>0}^p \zeta_{\ell_i}^{|A_i|}
   \int_{E^n} f(\mathbf{x}[\pi])\,\mu^{\otimes n}(d\mathbf{x}).
   \end{eqnarray}
   Note that, for a given partition with $p$ blocks, the integration
   appearing in the last line runs effectively only over $E^p$.
   For further simplification assume that the blocks $A_1,\ldots,A_p$
   of $\pi=\{A_1,\dots,A_p\}\in\mathcal{P}_n$ are
   enumerated according to decreasing block size, and write
   $s(\pi)$ for the number of singleton blocks of
   the partition $\pi=\{A_1,\dots,A_p\}$. Then, for a given
   $\pi=\{A_1,\dots,A_p\}\in\mathcal{P}_n$, the last sum in
   (\ref{eq:altform1}) can be written as
   $$
   \sum_{l=0}^{s(\pi)} {{s(\pi)}\choose l}
   (1-|\bzeta|)^{s(\pi)-l}
   \sum_{i_1,\dots,i_{p-s(\pi)+l}\in\N\atop\text{all distinct}}
   \hspace{-2em}
   \zeta_{i_1}^{|A_1|}\cdots\zeta_{i_{p-s(\pi)+l}}^{|A_{p-s(\pi)+l}|}
   \int_{E^n} f\big(\mathbf{x}[\pi]\big)\mu^{\otimes n}(d\mathbf{x}).
   $$
   Furthermore, for any $\bzeta\in\Delta$ and $n\in\N$,
   \begin{eqnarray*}
      1
      & = & \Big( \big(1-|\bzeta|\big) +
              {\textstyle\sum_{i=1}^\infty} \zeta_i \Big)^n\\
      & = & \sum_{\pi=\{A_1,\dots,A_p\}\in\mathcal{P}_n}
            \sum_{l=0}^{s(\pi)} {s(\pi)\choose l}
            (1-|\bzeta|)^{s(\pi)-l}
            \hspace{-2em}
            \sum_{i_1,\dots,i_{p-s(\pi)+l}\in\N\atop\text{all disticnt}}
            \hspace{-2em}
            \zeta_{i_1}^{|A_1|}\cdots\zeta_{i_{p-s(\pi)+l}}^{|A_{p-s(\pi)+l}|}.
   \end{eqnarray*}
   This allows us to re-express the inner integral in (\ref{eq:genXiFV1}) as
   \[
   \begin{aligned}
   \sum_{\pi=\{A_1,\dots,A_p\}\in\mathcal{P}_n} &
   \sum_{l=0}^{s(\pi)} {s(\pi) \choose l}
   (1-|\bzeta|)^{s(\pi)-l}
   \hspace{-2em}
   \sum_{i_1,\dots,i_{p-s(\pi)+l} \in \N \atop
     \text{all distinct}}
   \hspace{-2em}
   \zeta_{i_1}^{|A_1|}\cdots\zeta_{i_{p-s(\pi)+l}}^{|A_{p-s(\pi)+l}|}
   \int_{E^n} [f\big(\mathbf{x}[\pi]\big)
                      -f(\mathbf{x})] \,
                           \mu^{\otimes n}(d\mathbf{x}) \\[2ex]
   = & \sum_{\pi=\{ A_1,\dots,A_p\} \in \mathcal{P}_n \atop
                     \text{not all singletons}}
   \hspace{0em}
   \sum_{l=0}^{s(\pi)} {s(\pi) \choose l}
   (1-|\bzeta|)^{s(\pi)-l}
   \hspace{-2em}
   \sum_{i_1,\dots,i_{p-s(\pi)+l} \in \N \atop
     \text{all distinct}}
   \hspace{-2em}
   \zeta_{i_1}^{|A_1|}\cdots\zeta_{i_{p-s(\pi)+l}}^{|A_{p-s(\pi)+l}|} \\
   & \hspace{10em} {} \times \int_{E^n} [f\big(\mathbf{x}[\pi]\big)
                      -f(\mathbf{x})] \,
                           \mu^{\otimes n}(d\mathbf{x}),
   \end{aligned}
   \]
   because $\mathbf{x}[\{\{1\},\dots,\{n\}\}] = \mathbf{x}$.
   Integrating this equation over $\Delta$ with respect to
   the measure $(\bzeta,\bzeta)^{-1}\Xi_0$ yields (\ref{eq:genXiFV.alt}).
   Note that (see also \cite[p.~844]{S03})
   \begin{eqnarray*}
      &   & \hspace{-15mm}
      \sum_{\pi=\{A_1,\dots,A_p\}\in\mathcal{P}_n\atop\text{not all singletons}}
      \sum_{l=0}^{s(\pi)} {s(\pi)\choose l}
      (1-|\bzeta|)^{s(\pi)-l}
      \hspace{-2em}
      \sum_{i_1,\dots,i_{p-s(\pi)+l}\in\N\atop\text{all distinct}}
      \hspace{-2em}
      \zeta_{i_1}^{|A_1|}\cdots\zeta_{i_{p-s(\pi)+l}}^{|A_{p-s(\pi)+l}|}\\
      & \le &
      \sum_{\pi=\{A_1,\dots,A_p\}\in\mathcal{P}_n\atop\text{not all singletons}}
      \Big(\sum_{i_1=1}^\infty \zeta_{i_1}^2\Big)
      \sum_{l=0}^{s(\pi)} {s(\pi)\choose l}
      (1-|\bzeta|)^{s(\pi)-l}
      \hspace{-2em}
      \sum_{i_{p-s(\pi)+1},\dots,i_{p-s(\pi)+l}\in\N}
      \hspace{-2em}
      \zeta_{i_{p-s(\pi)+1}}\cdots\zeta_{i_{p-s(\pi)+l}}\\
      & = &
      \sum_{\pi=\{A_1,\dots,A_p\}\in\mathcal{P}_n\atop\text{not all singletons}}
      (\bzeta,\bzeta)\sum_{l=0}^{s(\pi)}
      {{s(\pi)}\choose l}(1-|\bzeta|)^{s(\pi)-l}|\bzeta|^l
      \ = \ (|\mathcal{P}_n|-1)\,(\bzeta,\bzeta)
   \end{eqnarray*}
   to verify that there is no singularity near $\bzeta=\mathbf{0}$.
\end{proof}

\subsection{Construction of the Markov semigroup and proof of Proposition~\ref{mainprop}}
\label{markov_semigroup}

The following proposition ensures that there exists a Markov process
attached to the $\Xi_{0}$-Fleming-Viot generator.
\begin{proposition} \label{prop:semigroup}
   The closure of $\{(G_f, L^{\Xi_{0}}G_f):n\in\N, f:E^n\to\R\;
   \mbox{bounded and measurable}\}$ generates a Markov semigroup
   on $\M_1(E)$.
\end{proposition}
\begin{proof}
   We write $G$ instead of $G_f$ for convenience. By the Hille-Yosida
   theorem (see, for example, \cite[p.~165, Theorem 2.2]{EK86}) it is
   sufficient to verify that
   \begin{enumerate}
      \item[(i)] the domain $D$ is dense in $C({\mathcal M}_1(E))$,
      \item[(ii)] the operator $L^{\Xi_0}$ satisfies the positive maximum
         principle, i.e., $L^{\Xi_0}G(\mu)\le 0$ for all $G\in D$,
         $\mu\in{\mathcal M}_1(E)$ with
         $\sup_{\nu\in{\mathcal M}_1(E)}G(\nu)=G(\mu)\ge 0$, and that
      \item[(iii)] the range of $\lambda-L^{\Xi_{0}}$ is dense in
         $C({\mathcal M}_1(E))$ for some $\lambda>0$.
   \end{enumerate}
   In order to verify (i) and (iii) we mimic the proof of Proposition~3.5
   in Chapter 1 of \cite{EK86} and construct a suitable sequence
   $D_1,D_2,\ldots$ of finite-dimensional subspaces of $C({\mathcal M}_1(E))$
   such that $D:=\bigcup_{k\in\N}D_k$ is dense in $C({\mathcal M}_1(E))$ and
   $L^{\Xi_{0}}:D_k\to D_k$ for all $k\in\N$ as follows. For $n\in\N$ and
   $f:E^n\to\R$ bounded and measurable let $D_f$ denote the set of all
   linear combinations of elements from the set
   \[
   \{G:G(\mu)={\textstyle\int f(\mathbf{x}[\pi])\,
   \mu^{\otimes n}(d\mathbf{x})},\pi\in{\mathcal P}_n\}.
   \]
   Since $|{\mathcal P}_n|<\infty$, it is easily seen that $D_f$ is a
   finite-dimensional subspace of $C(\M_1(E))$. From (\ref{eq:genXiFV.alt})
   it follows that $L^{\Xi_0}:D_f\to D_f$. For each $n\in\N$ let
   $\{g_{nm}:m\in\N\}\subset C(E^n)$ be dense, and let
   $\{f_k:k\in\N\}$ be an enumeration of $\{g_{nm}:n,m\in\N\}$. Then,
   $D_k:=D_{f_k}$, $k\in\N$, has the desired properties. Note that
   $D:=\bigcup_{k\in\N}D_k$ is dense in $C({\mathcal M}_1(E))$
   (Stone-Weierstrass), i.e. condition (i) holds.

   We have $(\lambda-L^{\Xi_{0}})(D_k)=D_k$ for all $\lambda$ not belonging to
   the set of eigenvalues of $L^{\Xi_{0}}|_{D_k}$, i.e., for all but at most finitely
   many $\lambda>0$. Thus, $(\lambda-L^{\Xi_{0}})(D)=(\lambda-L^{\Xi_{0}})(\bigcup_{k\in\N}D_k)
   =\bigcup_{k\in\N}D_k=D$ is dense in $C({\mathcal M}_1(E))$ for all but at
   most countably many $\lambda>0$. In particular, condition (iii) is
   satisfied.

   Condition (ii) follows from the fact that the expression inside the
   integrals in (\ref{eq_genXiFV_xi}) satisfies
   $$
   G((1-|\bzeta|)\mu+{\textstyle\sum_{i=1}^\infty\zeta_i\delta_{x_i}})
   - G(\mu)\ \le \sup_{\nu\in{\mathcal M}_1(E)}G(\nu)-G(\mu)
   \ =\ G(\mu)-G(\mu)\ =\ 0
   $$
   for all ${\mathbf x}=(x_1,x_2,\ldots)\in E^\N$, ${\bzeta}\in\Delta$,
   $G\in D$ and $\mu\in{\mathcal M}_1(E)$ with
   $\sup_{\nu\in{\mathcal M}_1(E)}G(\nu)=G(\mu)$.

   Thus, the Hille-Yosida theorem ensures that the closure
   $\overline{L^{\Xi_0}}$ of $L^{\Xi_0}$ on $C({\mathcal M}_1(E))$ is
   single-valued and generates a strongly continuous, positive, contraction
   semigroup $\{T_t\}_{t\ge 0}$ on ${\mathcal M}_1(E)$. Note that from (iii)
   it follows that $D$ is a core for $\overline{L^{\Xi_{0}}}$ (\cite[p.~166]{EK86}).
   The operator $L^{\Xi_{0}}$ maps constant functions to the zero function, i.e.,
   $L^{\Xi_{0}}$ is conservative. Thus, $\{T_t\}_{t\ge 0}$ is a Feller semigroup
   and corresponds to a Markov process % $Z=\{Z(t)\}_{t\ge 0}$
   with sample paths in $D_{{\mathcal M}_1(E)}([0,\infty))$.
\end{proof}
\begin{remark}\label{remark_generator}
 i) If the finite measure $\Xi$ on $\Delta$ allows for some mass
    $a:=\Xi(\{\mathbf 0\})$ at zero, then $L^{\Xi_{0}}$ has to be replaced by
    $L^\Xi := L^{\Xi_{0}} + L^{a\delta_{\bf 0}}$, where $L^{\Xi_{0}}$ is defined as before and $L^{a\delta_{\bf 0}}$ is the generator
    of the classical Fleming-Viot process \cite{FV79} given by \eqref{eq_genXiFV_kingman}.
    The existence of a Markov process $Z=(Z_t)_{t\ge 0}$ with generator
    $L^\Xi$ can be deduced as in the proof of Proposition \ref{prop:semigroup}
    via the Hille-Yosida theorem.\\
ii) The construction of the Markov process attached to the `full' generator $L$,
    including the Kingman component \eqref{eq_genXiFV_kingman} and the mutation
    component \eqref{eq_genXiFV_mutation}, works via the standard Trotter approach.\\
iii) Note that $\int (L^\Xi)G\,d\delta_{\delta_x}=0$, $x\in E$, where
   $\delta_{\nu}\in{\mathcal M}_1({\mathcal M}_1(E))$ denotes the
   unit mass at $\nu\in{\mathcal M}_1(E)$. Thus, see \cite[p.~239, Proposition
   9.2]{EK86}, the states $\delta_x$, $x\in E$, are absorbing for the $\Xi$-Fleming-Viot process.
\end{remark}

We now turn to the proof of Proposition~\ref{mainprop}. Indeed, we verify
the following

\vspace{2mm}

{\em Claim:}
   The distribution of the measure valued Markov process with generator
   $L$, as defined in Remark~\ref{remark_generator} ii), coincides with
   the distribution of the $(\Xi,B)$-Fleming-Viot process, as defined
   in Theorem~\ref{main}.

\vspace{2mm}

It suffices to verify the following lemma.
\begin{lemma} \label{thm_martingale_problem}
  The $(\Xi,B)$-Fleming-Viot process defined in Theorem~\ref{main} solves
  the martingale problem for the generator $L$ given in \eqref{eq_genXiFV}.
\end{lemma}

To prepare this, let us concentrate on the case when there is no mutation
and no Kingman-component ($L=L^{\Xi_0}$).
Fix $l$ and suppose we are at the $m$-th birth event. As in the
previous section, let $\{\phi_m^1,\ldots,\phi_m^{a_m}\}$ denote the
assignments of the levels to one of the $a_m$ families. So $\phi_m^i
\subset \{1,\ldots,l\}$ and $\phi_m^i\cap\phi_m^i\neq\emptyset$ for all
$i,j$. Furthermore, we again denote by $\Phi_m :=
\bigcup_{i=1}^{a_m}\phi_m^i$ all individuals participating in the birth
event. Note, that this can be a strict subset of $\{0,\ldots,l\}$, and
$\{\phi_m^1,\ldots,\phi_m^{a_m}\}$ holds all information about what is
going on at the birth event. The function $g(\bzeta,u)$ is
defined as in \eqref{eq_definition_g}.
We introduce a Poisson process counting the number of times a specific
birth event $\{\phi_m^1,\ldots,\phi_m^{a_m}\}$ happens. With
$(t_m,\bzeta_m,{\mathbf u}_m)$ denoting the points of the Poisson
point process $\mathfrak{M}^{\Xi_{0}}$ we define
\begin{equation}
   L_{\{\phi_m^1,\ldots,\phi_m^{a_m}\}}(t)
   \ :=\
   \sum_{t_m\le t} \sum_{b_1,\ldots,b_{a_m}\in\N \atop \text{all\,distinct}}
   \prod_{i=1}^{a_m} \prod_{j\in\phi^i_m} \1_{\{g(\bzeta_m,u_{mj})=b_i\}}
   \prod_{j \in \{1,\ldots,l\}\setminus\Phi_m} \1_{\{g(\bzeta_m,u_{mj})=\infty\}}.
\end{equation}

To describe the effect of the birth event
$\{\phi_m^1,\ldots,\phi_m^{a_m}\}$ on the population vector $x \in
E^l$ we introduce the function $\mathfrak{T}$ defined by
\begin{equation}
    \big(\mathfrak{T}_{\{\phi_m^1,\ldots,\phi_m^{a_m}\}} ({\mathbf x})\big)_i := \begin{cases}
                            x_{\min(\phi_m^j)}  &\text{if}\, k \in \phi_m^j,\\
                            x_{J_m(i)}          &\text{else}
                                                  \end{cases}
\end{equation}
for all $k \in \{1,\ldots,l\}$, where $J_m$ is the function defined in
\eqref{eq_function_jm} that holds the information on where the
non-participating particles should look down to.

With this notation we can use equation \eqref
{eq_stochastic_differential_equation} and the dependence between the
$L_{J,k}^l$ and $L_{J}^l$ to show that
\begin{equation}\label{eq_level_stochastic_differential_equation}
    X^l(t) := X^l(0) + \sum_{\{\phi_m^1,\ldots,\phi_m^{a_m}\}, \atop \dot{\bigcup}\phi_m^i \subset \{1,\ldots,l\}}
\int_0^t \Big( \mathfrak{T}_{\{\phi_m^1,\ldots,\phi_m^{a_m}\}}\big(X^l(s-)\big) - X^l(s-) \Big)\, dL_{\{\phi_m^1,\ldots,\phi_m^{a_m}\}}(s)
\end{equation}
describes the evolution of the first $l$ levels $X^l \in E^l$, if we
assume no mutation and no Kingman part. Note that for simplicity we
use the notation $X^l = (X_1, \ldots, X_l)$.

Since the $L_{\{\phi_m^1,\ldots,\phi_m^{a_m}\}}(t)$ are Poisson processes
derived from the Poisson point process $\mathfrak{M}^{\Xi_{0}}$ it is
straightforward to verify that their rates are given by
\begin{equation} \label{eq_rates_of_poisson}
   r\big(\{\phi_m^1,\ldots,\phi_m^{a_m}\}\big)
   \ :=\ \sum_{i_1,\ldots,i_{a_m} \atop \text{all\ distinct}}
   \int_\Delta
   \zeta_{i_1}^{k_m^1+1}\cdots\zeta_{i_r}^{k_m^r+1}
   \zeta_{i_{r+1}}\cdots\zeta_{i_{a_m}}(1-|\bzeta|)^{(l-|\Phi|)}
   \frac{\Xi_0(d\bzeta)}{(\bzeta,\bzeta)},
\end{equation}
where $k_m^i+1=|\phi_m^i|$ as before and the sets are ordered, such that
$k_m^1\ge\cdots\ge k_m^r\ge 1$ and $k_m^{r+1}=\cdots=k_m^{a_m}=0$
hold. Assume that at least $k_m^1\ge 1$ holds, because otherwise
$\mathfrak{T}$ is the identity. Note that under this assumption the
integral in \eqref{eq_rates_of_poisson} is finite (c.f.\
\cite{S00} or \cite{S03}).

We now turn to the actual proof of the lemma.

\begin{proof}[Proof of Lemma~\ref{thm_martingale_problem}]
We will prove the result for the generator $L^{\Xi_0}$. The full result can
then be obtained in analogy to the proof of Theorem~2.4 in \cite{DK96}.

Indeed, we have to show that for each function $G_f\in\mathcal{D}(L^{\Xi_0})$
of the form
\begin{equation}
    G_f(\mu)\ =\ \langle f,\mu^{\otimes l}\rangle,
\end{equation}
for $\mu\in\mathcal{M}_1(E)$ and $f\colon E^l \to \R$ bounded and measurable,
\begin{equation} \label{martingale}
    G_f(Z(t)) - G_f(Z(0)) - \int_0^t (L^{\Xi_0}G_f)(Z(s))\,ds
\end{equation}
is a martingale with respect to the natural filtration of the Poisson point process $\mathfrak{M}^{\Xi_{0}}$ given by
\begin{equation}
    \{\mathcal{J}_t\}_{t\ge 0} := \Big\{ \sigma\big( \mathfrak{M}^{\Xi_{0}}\Big\vert_{[0,t]\times\Delta\times[0,1]^\N}\big) \Big\}_{t\ge 0}.
\end{equation}
Note that
\begin{equation}\label{f_equals_b}
  \mathbb{E}\Big[ f\big( X_1(s),\ldots,X_l(s)\big)\Big\vert \mathcal{J}_t \Big]
= \mathbb{E}\Big[ \big\langle f,Z(s)^{\otimes l}\big\rangle \Big\vert \mathcal{J}_t\Big]
\end{equation}
holds for all $s,t\ge 0$, which will be crucial in the following steps.

We start by observing that, for $0\le w\le t$, the representation \eqref{eq_level_stochastic_differential_equation} leads to
\begin{multline}\label{eq_stochastic_martingale}
  0 = \mathbb{E} \Big[ f\big(X^l(t)\big) - f\big(X^l(w)\big)\\
  - \sum_{\{\phi_m^1,\ldots,\phi_m^{a_m}\}, \atop \dot{\bigcup}\phi_m^i \subset \{1,\ldots,l\}}\int_w^t \Big(
  f\Big(\mathfrak{T}_{\{\phi_m^1,\ldots,\phi_m^{a_m}\}}\big(X^l(s)\big)\Big)
  - f\big(X^l(s)\big)
  \Big)r\big(\{\phi_m^1,\ldots,\phi_m^{a_m}\}\big)\, ds \Big\vert
  \mathcal{J}_w \Big],
\end{multline}
since this is a martingale.

Using the definition of the rates \eqref{eq_rates_of_poisson} and the
fact that due to the exchangeability of $X^l$, the action of
$\mathfrak{T}_{\{\phi_m^1,\ldots,\phi_m^{a_m}\}}$ and the $[\pi]$
operation under the expectation is the same, we can now rewrite the
last term (without the substraction of $f(X^l(s))$ from the integrand) as
\begin{eqnarray}
  &   & \hspace{-8mm}
        \mathbb{E} \bigg[ \int_w^t \sum_{\{\phi_m^1,\ldots,\phi_m^{a_m}\}, \atop \dot{\bigcup}\phi_m^i \subset \{1,\ldots,l\}} r\big(\{\phi_m^1,\ldots,\phi_m^{a_m}\}\big) f\Big(\mathfrak{T}_{\{\phi_m^1,\ldots,\phi_m^{a_m}\}}
\big(X^l(s)\big)\Big)\, ds \bigg| \mathcal{J}_w \bigg]\nonumber\\
  & = & \mathbb{E} \bigg[ \int_w^t \sum_{\pi=\{A_1,\dots,A_p\} \in \mathcal{P}_n} \sum_{(r_1,\dots,r_p) \atop
\text{admissible}}  \int_\Delta (1-|\bzeta|)^{\#\{r_i=0\}} \prod_{i=1 \atop r_i>0}^p \zeta_{r_i}^{|A_i|}
\frac{\Xi_0(d\bzeta)}{(\bzeta,\bzeta)} f\Big( \big(X^l(s)\big)[\pi]\Big)\,ds \bigg| \mathcal{J}_w \bigg]\nonumber\\
  & = & \mathbb{E} \bigg[ \int_w^t \int_\Delta \sum_{\pi=\{A_1,\dots,A_p\} \in \mathcal{P}_n} \sum_{(r_1,\dots,r_p)
\atop \text{admissible}}   (1-|\bzeta|)^{\#\{r_i=0\}} \prod_{i=1 \atop r_i>0}^p \zeta_{r_i}^{|A_i|} \langle f\circ [\pi],Z(s)^{\otimes l}
\rangle \frac{\Xi_0(d\bzeta)}{(\bzeta,\bzeta)} \,ds \bigg| \mathcal{J}_w \bigg]\nonumber\\
  & = & \mathbb{E} \bigg[ \int_w^t \int_\Delta \int_{E^\N}
  G_f\big( (1-|\bzeta|)Z(s) + {\textstyle\sum_{i=1}^\infty \zeta_i
    \delta_{x_i}}\big) Z(s)^{\otimes
    \N}(d\mathbf{x})\frac{\Xi_0(d\bzeta)}{(\bzeta,\bzeta)} \,ds\bigg| \mathcal{J}_w \bigg],
    \label{eq_really_weird_stuff}
\end{eqnarray}
since the sum about the configurations $\{\phi_m^1,\ldots,\phi_m^{a_m}\}$ and the distinct indices $i_1,\ldots,i_{a_m}$ can be
rewritten as the sum about the partitions $\pi$ and the admissible vectors $(r_1,\ldots,r_p)$. The
last equality holds due to equation \eqref{eq:altform1}.

Combining equation \eqref{eq_stochastic_martingale} with equation
\eqref{eq_really_weird_stuff} we see that
\begin{eqnarray}
   0
   & = & \mathbb{E}\bigg[
            f\big(X^l(t)\big) - f\big(X^l(w)\big)\nonumber\\
   &   & - \int_w^t \int_\Delta \int_{E^\N}
              \big(G_f\big((1-|\bzeta|)Z(s) + {\textstyle\sum_{i=1}^\infty \zeta_i
              \delta_{x_i}}\big) - G_f(Z(s))\big)
              Z(s)^{\otimes \N}(d\mathbf{x})\frac{\Xi_0(d\bzeta)}{(\bzeta,\bzeta)} \,ds
            \bigg|\mathcal{J}_w
        \bigg]\nonumber\\
  & = & \mathbb{E}\bigg[
           \langle f,Z(t)^{\otimes l}\rangle -
           \langle f,Z(w)^{\otimes l}\rangle -
           \int_w^t (L^{\Xi_{0}}G_f)(Z(s))\,ds\bigg|\mathcal{J}_w
        \bigg]\nonumber\\
   & = & \mathbb{E}\bigg[
            G_f\big(Z(t)\big) - G_f\big(Z(w)\big)
            - \int_w^t(L^{\Xi_0}G_f)(Z(s))\,ds\bigg|\mathcal{J}_w
         \bigg]
\end{eqnarray}
holds, where we use \eqref{f_equals_b} in the second equality. Thus,
(\ref{martingale}) is a martingale.
\end{proof}

\section{Dualities}\label{dualities}

\subsection{Distributional duality versus pathwise duality}\label{sec_distributional_duality}

We first establish a {\em distributional duality} in the classical sense of \cite{L85}.
Indeed, (\ref{eq:genXiFV.alt}) and results about the classical Fleming-Viot process bring forth the
following duality between a $\Xi$-coalescent $\Pi=(\Pi_t)_{t\ge 0}$
and a $\Xi$-Fleming-Viot process $Z=(Z_t)_{t\ge 0}$.

\begin{lemma} (Duality) \label{lemma:duality}
   For $n\in\N$, $f:E^n\to\R$ bounded and measurable,
   $\mu\in{\mathcal M}_1(E)$, $\pi\in{\mathcal P}_n$ and $t\ge 0$,
   \begin{equation}\label{analyticduality}
   \mathbb{E}^\mu \Big[ \int_{E^n}
   f\big(\mathbf{x}[\pi]\big) Z_t^{\otimes n}(d\mathbf{x}) \Big]
   =
   \mathbb{E}^\pi \Big[ \int_{E^n}
   f\big(\mathbf{x}[\Pi_t^{(n)}]\big) \mu^{\otimes n}(d\mathbf{x}) \Big],
   \end{equation}
   where $\Pi_t^{(n)}$ is the restriction of $\Pi_t$ to ${\mathcal P}_n$.
\end{lemma}

To obtain a {\em pathwise duality}, we use the driving Poisson
processes of the modified lookdown construction to construct
realisation-wise a $\Xi$-coalescent embedded in the $\Xi$-Fleming-Viot
process.

More explicitly, recall the Poisson processes $L^l_J$ and $L^l_{J,k}$
from equation \eqref{eq_llj} and \eqref{eq_lljk} in Section~\ref{limitpop}
and the Poisson process $\mathfrak{N}^K_{ij}$ defined in Section~\ref{ExCoal}.
For each $t\ge 0$ and $l\in\N$, let $N_t^l(s), 0\le s\le t$, be the level
at time $s$ of the ancestor of the individual at level $l$ at time $t$.
In terms of the $L^l_J$ and $L^l_{J,k}$, the process $N^l_t(\cdot)$ solves,
for $0\le s\le t$,

\begin{align}
\label{genealllogy}
N_t^l(s)\ =\ l &- \sum_{1 \le i < j < l} \int_{s-}^t \1_{\{N_t^l(u+) > j\}} \, d\mathfrak{N}^K_{ij}(u) \notag \\
         & -  \sum_{1 \le i < j < l} \int_{s-}^t (j-i) \1_{\{N_t^l(u+) =  j\}} \, d\mathfrak{N}^K_{ij}(u) \notag \\
    & - \sum_{K \subset \{1,\ldots,l\}} \int_{s-}^t (N_t^l(u+) - J_m(N_t^l(u+)))\1_{\{N_t^l(u+) \notin K\}} \, dL^l_K(u)\notag\\
    & - \sum_{k \in \N} \sum_{K \subset \{1,\ldots,l\}} \int_{s-}^t (N_t^l(u+)-\min(K))
                 \1_{\{N_t^l(u+) \in K\}}\,dL^l_{K,k}(u),
\end{align}
where $J_m(\cdot)=J_{m(u)}(\cdot)$ is defined by \eqref{eq_function_jm} and $m(u)$ is the index of the jump at time $u$.
Fix $0\le T$ and, for $t\le T$, define a partition
$\Pi_t^T$ of $\N$ such that $k$ and $l$ are in the same block
of $\Pi_t^T$ if and only if $N_T^l(T-t) = N_T^k(T-t)$.  Thus, $k$
and $l$ are in the same block if and only if the two levels $k$ and
$l$ at time $T$ have the same ancestor at time $T-t$.
Then (\cite{DK99}, Section~5),
\begin{equation}
\label{eq:pathwisedual}
\mbox{the process $(\Pi_t^T)_{0\le t\le T}$
is a $\Xi$-coalescent run for time $T$}.
\end{equation}
Note that by employing a natural generalisation of the lookdown
construction using driving Poisson processes on $\R$ and e.g.\
using $T=0$ above, one can use the same construction to find an
$\Xi$-coalescent with time set $\R_+$.  We would like to
emphasise that the lookdown construction provides a realisation-wise
coupling of the type distribution process $(Z_t)_{t\ge 0}$ and the
coalescent describing the genealogy of a sample, thus extending
\eqref{analyticduality}, which is merely a statement about
one-dimensional distributions.

\subsection{The function-valued dual of the $(\Xi,B)$-Fleming-Viot process}

The duality between the $   \Xi$-Fleming-Viot process and the
$\Xi$-coalescent established in
Section~\ref{sec_distributional_duality} worked only on the
genealogical level, the mutation was not taken into account. However,
it is possible to define a function-valued dual to the
$(\Xi,B)$-Fleming-Viot process such that not only the genealogical
structure, but also the mutation is part of the duality. This kind of
duality is well known for the classical Fleming-Viot
process, see, e.g., Etheridge \cite[Chapter 1.12]{E00}.
%%and we go along the same lines.

First note, that due to Lemma~\ref{lemma_generator_alt} we can rewrite
the generator of the $(\Xi,B)$-Fleming-Viot process given by equation
\eqref{eq_genXiFV} to obtain
\begin{align}
LG_f(\mu) &:=  a\sum_{1\le i < j \le n} \int_{E^n} \Big(f(x_1,\!.., x_i,\!.., x_i,\!.., x_n)- f(x_1,\!.., x_i,\!.., x_j,\!.., x_n)
                    \Big)\mu^{\otimes n}(d{\bf x}) \nonumber\\
              &\quad + \sum_{\pi = \{ A_1,\dots,A_p\} \in \mathcal{P}_n \atop \text{ not all singletons}}
                 \lambda(|A_1|,\dots,|A_p|)
                 \int_{E^n} \left( f\big(\mathbf{x}[\pi]\big)
                                  -f(\mathbf{x})\right)
                           \mu^{\otimes n}(d{\bf x}), \nonumber\\
              &\quad + r \sum_{i=1}^n \int_{E^n} B_i(f(x_1, \dots, x_n)) \mu^{\otimes n}(d{\bf x}). \label{eq_real_alt_generator}
\end{align}
We can now reinterpret the function $G_f(\mu)$ acting on measures as a
function $G_\mu(f)$ acting on the functions $C_b(E^n)$. This
reinterpretation transfers the operator $L$ acting on
$C\big(\M_1(E)\big)$ to an operator $L^*$ acting on
$C_b\big(C_b(E^n)\big)$. %%We want to look at the Markov process
%%generated by $L^*$. Therefore, l
Let $\mathcal{C} := \bigcup_{n=1}^\infty C_b(E^n)$. A
$\mathcal{C}$-valued Markov process $(\rho_t)_{t\ge 0}$ solving
the martingale problem for $L^*$ can then be constructed as follows:
\begin{itemize}
\item If $\rho_t({\bf x})\in C_b(E^n)$ and $n\ge 2$, then the
  process $(\rho_t)_{t\ge 0}$ jumps to $\rho_t\big({\bf x[\pi]}\big)$ with rate
  $\lambda(|A_1|,\dots,|A_p|) + a\1_{\{\exists!|A_i| =2;\forall j\neq i: |A_j|=1\}}$,
  for all $\pi = \{ A_1,\dots,A_p\} \in \mathcal{P}_n$, where $|A_j|
  \ge 1$ for at least one $j$.
\item If $\rho_t\in C_b(E)$, that is it is a function of a single
  variable, then no further jumps occur.
\item Between jumps the process evolves deterministically
  according to the ``heat flow'' generated by the mutation operator \eqref{mutop},
  independently for each coordinate.
\end{itemize}
Note that this process is not literally a coalescent, but has
coalescent-like features.

The duality relation between $\rho_t$ and $Z_t$ immediately follows
from \eqref{eq_real_alt_generator} and can be written in integrated
form as
\begin{equation}
   \E_{Z_0} \langle \rho_0, Z_t^{\otimes n} \rangle
\ =\ \E_{\rho_0} \langle \rho_t , Z_0^{\otimes n} \rangle.
\end{equation}
It can be used for example to show uniqueness of the martingale problem
for $L$ via the existence of $(\rho_t)_{t\ge 0}$ or to calculate the
moments of the $(\Xi,B)$-Fleming-Viot process.

\subsection{The dual of the block counting process}

In this section, we specialise to the case where the type space $E$
consists of two types only, say $E=\{0,1\}$. Define the real-valued process
$Y=(Y_t)_{t\ge 0}$ via
$Y_t:=Z_t(\{\text{1}\})$, $t\ge 0$.  Define $g:{\mathcal M}_1(E)\to [0,1]$ via
$g(\mu):=\mu(\{1\})$.  The generator $A$ of $Y$ is then
given by $Af(x)=(L^\Xi(f\circ g))(\mu)$, $f\in C^2([0,1])$, where
$\mu$ depends on $x\in[0,1]$ and can be chosen arbitrary, as long as
$g(\mu)=x$. Thus,
\begin{equation} \label{eq:genA}
Af(x)\ =\ a\frac{x(1-x)}{2}f''(x) + \int_\Delta\int_{\{0,1\}^\N}
\big(
   f((1-|\bzeta|)x+{\textstyle\sum_{i=1}^\infty} \zeta_i y_i) - f(x)
\big)(\mathcal{B}(1,x))^{\otimes\N}(d{\mathbf y})
\frac{\Xi_{0}(d\bzeta)}{(\bzeta,\bzeta)},
\end{equation}
$x\in [0,1]$, $f\in C^2([0,1])$, where $\mathcal{B}(1,x)$ denotes the
Bernoulli distribution with parameter $x$. For $x\in[0,1]$ let
$V_1(x),V_2(x),\ldots$ be a sequence of independent and identically
$\mathcal{B}(1,x)$-distributed random variables. Then,
\[
Af(x)\ =\ a\frac{x(1-x)}{2}f''(x) + \int_\Delta\int_{[0,1]}
\big(
   f((1-|\bzeta|)x+y) - f(x)
\big)Q(\bzeta,x,dy)
\frac{\Xi_0(d\bzeta)}{(\bzeta,\bzeta)},
\]
where $Q(\bzeta,x,.)$ denotes the distribution of
$\sum_{i=1}^\infty \zeta_iV_i(x)$. Hence the process can be considered as
a Wright-Fisher diffusion with jumps. The situation where $\Xi$ is
concentrated on $[0,1]\times\{0\}^\N$, i.e., when the underlying
$\Xi$-coalescent is a $\Lambda$-coalescent, has been studied in \cite{BLG05}.

Note that $A f\equiv 0$ for $f(x)=x$, so $Y$ is a martingale. Furthermore,
the boundary points $0$ and $1$ are obviously absorbing.

In analogy to Lemma \ref{lemma:duality}
it follows that $Y$ is dual to the block counting process
$D=(D_t)_{t\ge 0}$ of the $\Xi$-coalescent with respect to the
duality function $H:[0,1]\times\N\to\R$,
$H(x,n):=x^n$ (see, e.g., Liggett \cite{L85}), i.e.,
\[
\mathbb{E}^y[Y_t^n]\ =\ \mathbb{E}^n[y^{D_t}],\quad n\in\N,y\in [0,1],t\ge 0.
\]
Thus, the moments of the `forward' variable $Y_t$ can be computed via
the generating function of the `backward' variable $D_t$ and vice versa.
Such and closely related moment duality relations are well known from the
literature \cite{AH07, AS05, M99}. The duality can be used to relate the
accessibility of the boundaries of $Y$ and the existence of an entrance
law for $D$ with $D_{0+}=\infty$. Note that by the Markov property and
the structure of the jump rates, we always have
\begin{equation} \label{eq:Dthits1}
   \P^\infty(D_t=1\mbox{ eventually})\ \in\ \{0,1\}
\end{equation}
and either $\P^\infty(\bigcap_{t\ge 0}\{D_t=\infty\})=1$
(if the probability in (\ref{eq:Dthits1}) equals $0$) or
$\lim_{t\to\infty}\P^\infty(D_t=1)=1$ (if the probability in
(\ref{eq:Dthits1}) equals $1$).
\begin{proposition} \label{prop:hitbdry}
   $\lim_{t\to\infty}\P^\infty(D_t=1)=1$ if and only if $Y$, the dual
   of its block counting process, hits the boundary $\{0,1\}$ in finite
   time almost surely, starting from any $y\in(0,1)$.
\end{proposition}
\begin{proof}
   Fix $y\in(0,1)$, $T>0$. Construct $(Z_t)$ starting from
   $y\delta_1+(1-y)\delta_0$ and no mutations, $Bf\equiv 0$, (and
   hence $Y$ starting from $y$) by using the lookdown construction
   from Section~\ref{limitpop}: Let $X_1(0),X_2(0),\dots$ be independent
   $\mathcal{B}(1,y)$-distributed random variables which are independent
   of the driving Poisson processes, and let $X_n(t)$, $t>0$, $n\in\N$,
   be the solution of (\ref{eq_stochastic_differential_equation}). Let
   \[
   D'_t\ :=\ |\{N^n_T(T-t):n\in\N\}|,
   \]
   where $N^n_T(s)$ solves (\ref{genealllogy}). By
   (\ref{eq:pathwisedual}), the law of $(D'_t)_{0 \le t \le T}$ is that
   of the block counting process of the (standard-)$\Xi$-coalescent run
   for time $T$. Then by construction (as there is no mutation),
   \[
   X_n(T) = X_{N^n_T(0)}(0),
   \]
   implying
   \[
   \{D'_T=1\} \subset \{Y_T\in\{0,1\}\}
   \quad\mbox{and}\quad
   \{D'_T=\infty \}\subset \{0<Y_T<1\}\mbox{ almost surely},
   \]
   which easily yields the claim.
\end{proof}
This is related to the so-called `coming down from infinity'-property
of the standard $\Xi$-coalescent (i.e., the property that starting from
$D_0=\infty$, $D_t<\infty$ almost surely for all $t>0$). Recall
(\cite{S00}, p.~39f) that a $\Xi$-coalescent may have infinitely many
classes for a positive amount of time and then suddenly jumps to
finitely many classes. This can occur if $\Xi$ has positive mass on
$\Delta_f:=\{{\mathbf u}=(u_1,u_2,\ldots)\in\Delta:u_1+\cdots+u_n=1
\mbox{ for some $n\in\N$}\}$. On the other hand \cite[Lemma~31]{S00},
if $\Xi(\Delta_f)=0$, then the $\Xi$-coalescent either comes down from
infinity immediately or always has infinitely many classes. Combining
this with Proposition~\ref{prop:hitbdry} we obtain
\begin{remark}
   Assume that $\Xi(\Delta_f)=0$. Then the $\Xi$-coalescent comes down
   from infinity if and only if the dual of its block counting process
   hits the boundary $\{0,1\}$ in finite time almost surely.
\end{remark}
In general, there seems to be no `simple' criterion to check whether a
$\Xi$-coalescent comes down from infinity (see the discussion in Section~5.5
of \cite{S00}). On the other side, there seems to be also no `handy'
criterion for accessibility of the boundary of a process with jumps (and with
values in $[0,1]$), but at least Proposition~\ref{prop:hitbdry}
allows to transfer any progress from one side to the other and vice versa.

We conclude this section with a simple toy example for which most
quantities of interest, in particular the generator $A$, can be
computed explicitly.
\begin{example}
   Fix $l\in\N$. If the measure $\Xi$ is concentrated on
   $\Delta_l:=\{\bzeta\in\Delta:\zeta_1+\cdots+\zeta_l=1\}$, then
   (\ref{eq:genA}) reduces to
   $$
   Af(x)\ =\ \int_\Delta
   \sum_{y_1,\ldots,y_l\in\{0,1\}}
   x^{y_1+\dots+y_l}(1-x)^{l-(y_1+\dots+y_l)}
   \big(f({\textstyle\sum_{i=1}^l \zeta_i y_i})-f(x)\big)
   \frac{\Xi(d\bzeta)}{(\bzeta,\bzeta)}.
   $$
   For example, assume that the measure $\Xi$ assigns its total mass
   $\Xi(\Delta):=1/l$ to the single point $(1/l,\ldots,1/l,0,0,\ldots)
   \in\Delta_l$. Then,
   $$
   Af(x)\ =\ \sum_{k=0}^l {l\choose k}x^k(1-x)^{l-k}f(k/l) - f(x)
   \ =\ \int (f(y/l)-f(x))\,{\cal B}(l,x)(dy),
   $$
   where ${\cal B}(l,x)$ denotes the binomial distribution with
   parameters $l$ and $x$. Note that the corresponding $\Xi$-coalescent
   never undergoes more than $l$ multiple collisions at one time. The
   rates (\ref{eq:xi.rate}) are
   $$
   \lambda(k_1,\ldots,k_p)
   \ =\ \int_\Delta
   \sum_{i_1,\dots,i_p\in\N\atop
     \text{all distinct}}
   \zeta_{i_1}^{k_1}\cdots \zeta_{i_p}^{k_p}\,
   \frac{\Xi(d\bzeta)}{(\bzeta,\bzeta)}
   \ =\ \frac{(l)_p}{l^n},
   $$
   where $(l)_p:=l(l-1)\cdots(l-p+1)$ and $n:=k_1+\cdots+k_p$. The block
   counting process $D$ has rates
   $$
   g_{np}\ =\ \frac{n!}{p!}
   \hspace{-2mm}
   \sum_{{{m_1,\ldots,m_p\in\N}\atop{m_1+\cdots+m_p=n}}}
   \hspace{-1em}
   \frac{\lambda(m_1,\ldots,m_p)}{m_1!\cdots m_p!}
   \ =\ S(n,p)\frac{(l)_p}{l^n},\quad 1\le p<n,
   $$
   where the $S(n,p)$ denote the Stirling numbers of the second kind.
   The total rates are $g_n=\sum_{p=1}^{n-1}g_{np}=1-(l)_n/l^n$,
   $n\in{\mathbb N}$. Note that the corresponding $\Xi$-coalescent
   stays infinite for a positive amount of time
   (`Case 2' on top of \cite[p.~39]{S00} with $\Xi_2\equiv 0$).
   The dual of its block counting process hits the boundary in
   finite time.\hfill$\blacksquare$
\end{example}
\section{Examples}\label{examples}
The first of the two examples in this section presents a model, where
the population size varies substantially due to recurrent bottlenecks.
It is shown, that the $\Xi$-coalescent appears naturally as the limiting
genealogy of this model. In the second example we present the
Poisson-Dirichlet-coalescent by choosing a particular measure for
$\Xi$ which has a density with respect to the
Poisson-Dirichlet distribution. We provide explicit expressions for
several quantities of interest.
\subsection{An example involving recurrent bottlenecks}
Consider a population, say with non-overlapping generations, in which the
population size has undergone occasional abrupt changes in the past.
Specifically, we assume that `typically', each generation contains $N$
individuals, but at several instances in the past, it has been substantially
smaller for a certain amount of time, and then the population has quickly
re-grown to its typical size $N$. This is related to the models considered
by Jagers \& Sagitov in \cite{JS04}, but we assume occasional much more
radical changes in population size than \cite{JS04}.
Let us assume that the demographic history is described by three
sequences of positive real numbers $(s_{i})_{i\in \N}$,
$(l_{i,N})_{i\in \N}$ and $(b_{i,N})_{i\in \N}$, where $0 < b_{i,N}
\le 1$ holds for all $i$, and the population size $t$ generations
before the present is given by $G(t)$, where
\[
   G(t)\ =\
   \begin{cases}
        b_{m,N}N & \text{if}\ N\big(\sum_{i=1}^{m-1}(s_i+l_{i,N})+s_m\big)<t\le N\sum_{i=1}^m (s_i+l_{i,N}),\quad m\in\N,\\
        N & \text{otherwise}.
   \end{cases}
\]
Thus, back in time the population stays at size $N$ for some time $s_i
N$. Then the size is reduced to $b_{i,N} N$ for the time $l_{i,N} N$.
Thereafter it is again given by $N$, until the next bottleneck occurs
after time $s_{i+1} N$. Note that for simplicity, we have assumed
`instantaneous' re-growth after each bottleneck. Furthermore, we
assume that the reproduction behaviour is given by the standard
Wright-Fisher dynamics, so each individual chooses its parent
uniformly at random from the previous generation, independently of the
other individuals. This is the case in every generation, also during
the bottleneck and at the transitions between the bottleneck and the
typical size.

We now want to keep track of the genealogy of a sample of $n$
individuals from the present generation, and describe its dynamics
in the limit $N\to\infty$. Denote by $\Pi^{(N,n)}(t)$ the ancestral
partition of the sample $t$ generations before the present.
\begin{lemma}\label{lemma:bottleneck}
    Fix $(s_i)_{i\in\N}$ and assume that $b_{i,N}\to 0$ and that
    $l_{i,N}\to 0$ as $N\to\infty$. Furthermore assume that
    $b_{i,N}N\to\infty$ and that $l_{i,N}/b_{i,N}\to\gamma_i>0$. Then
    \[
       \Pi^{(N,n)}(Nt)\ \to\ \Pi^{\delta_{\bf 0},(n)}(R_t)
    \]
    weakly as $N\to\infty$ on $D_{\mathcal{P}_n}([0,\infty))$, where
    $R_t:=t+\sum_{i:s_1+\cdots+s_i\le t}\gamma_i$.
\end{lemma}
Note that we assume $l_{i,N} \to 0$ as $N \to \infty$, so the duration
of the bottleneck is negligible on the timescale of the `normal'
genealogy. We also assume $b_{i,N}\to 0$ but $Nb_{i,N}\to\infty$,
i.e., in the pre-limiting scenario, the population size during a
bottleneck should be tiny compared to the normal size, but still large
in absolute numbers. The ratio $l_{i,N}/b_{i,N}$ is sometimes called
the {\em severity} of the ($i$-th) bottleneck in the population genetic
literature.

\begin{proof}[Sketch of proof]
  Given sequences $(s_i), (b_{i, N})$ and $(l_{i,N})$, classical
  convergence results for samples of size $n$ can be applied for the
  time-intervals between bottlenecks and ``inside'' the bottlenecks.
  Since $b_{i,N} N \to \infty$, the probability that any of the
  ancestral lines of the sample converge exactly at the transition to
  a bottleneck is $O((b_{i,N} N)^{-1})=o(1)$, so that na\"ive ``glueing'' is
  feasible.
\end{proof}

\begin{remark}
   Note that bottleneck events with $\gamma_i=0$ become invisible in
   the limit, whereas in a bottleneck with $\gamma_i=+\infty$ the
   genealogy necessarily comes down to only one lineage (and thus,
   all genetic variability is erased).
\end{remark}
Since we fixed the $s_i$ and $\gamma_i$, the limiting process
described in Lemma~\ref{lemma:bottleneck} is not a homogeneous Markov
process and thus does not fit literally into the class of exchangeable
coalescent processes considered in this paper. Assume that the waiting
intervals $s_i$ are exponentially distributed, say with parameter
$\beta$, and that the $\gamma_i$ are independently drawn from a
certain law $\mathcal{L}_\gamma$. Thus, in the pre-limiting
$N$-particle model forwards in time, in each generation there is a
chance of $\sim\beta/N$ that a `bottleneck event' with a randomly
chosen severity begins. In this situation, the genealogy of an
$n$-sample from the population at present is (approximately) described
by
\begin{equation}\label{eq:bottleneckcoal}
   \Pi^{\delta_{\bf 0},(n)}(S_t),\quad t\ge 0,
\end{equation}
where $(S_t)_{t\ge 0}$ is a subordinator (in fact, a compound Poisson process
with L\'evy measure $\beta \mathcal{L}_\gamma$ and drift $1$).
\begin{proposition}
   Let $N_\gamma$ be the number of lineages at time $\gamma>0$ in the
   standard Kingman coalescent starting with $N_0=\infty$, and let
   $D_j$ be the law of the re-ordering of a ($j$-dimensional)
   Dirichlet$(1,\dots,1)$ random vector according to decreasing size,
   padded with infinitely many zeros. The process defined in
   (\ref{eq:bottleneckcoal}) is the $\Xi$-coalescent
   restricted to $\{1,\ldots,n\}$, where
   \[
      \Xi(d{\bzeta})\ =\ \delta_{\bf 0}(d{\bzeta})
      + ({\bzeta},{\bzeta})\int_{(0,\infty)}
      \sum_{j=1}^\infty \P(N_\sigma=j) D_j(d\bzeta)
      \,\beta\mathcal{L}_\gamma(d\sigma).
   \]
\end{proposition}
\begin{proof}
  Recall that the number of families of the classical Fleming-Viot
  process without mutation after $\sigma$ time units is $N_\sigma$.
  Given $N_\sigma=j$, the distribution of the family sizes is a
  uniform partition of $[0,1]$, hence Dirichlet$(1,\dots,1)$.
  Size-ordering thus leads to the above formula for $\Xi$.
\end{proof}

\subsection{The Poisson-Dirichlet case}

The Poisson-Dirichlet distribution ${\rm PD}_\theta$ with parameter
$\theta>0$ is a distribution concentrated on the subset $\Delta^*$ of
points $\bzeta\in\Delta$ satisfying $|\bzeta|=1$. It can, for example,
be obtained via size-ordering of the normalized jumps of a
Gamma-subordinator at time $\theta$. For more information on this
distribution we refer to \cite{K75} or \cite{ABT99}.  Sagitov
\cite{S03} considered the Poisson-Dirichlet coalescent
$\Pi=(\Pi_t)_{t\ge 0}$ with parameter $\theta>0$, where (by
definition) the measure $\Xi$ has density
$\bzeta\mapsto(\bzeta,\bzeta)$ with respect to ${\rm PD}_\theta$. As
the measure ${\rm PD}_\theta$ is concentrated on $\Delta^*$, the rates
(\ref{eq:xi.rate}) reduce to
   $$
   \lambda(k_1,\ldots,k_j)\ =\ \int_{\Delta^*}
   \sum_{i_1,\dots,i_j\in\N\atop\text{all distinct}}
   \zeta_{i_1}^{k_1}\cdots\zeta_{i_j}^{k_j}\,{\rm PD}_\theta(d\bzeta).
   $$
   From the calculations of Kingman \cite{K93} it follows that
   the Poisson-Dirichlet coalescent has rates
   $$
   \lambda(k_1,\ldots,k_j)\ =\ \frac{\theta^j}{[\theta]_k}\prod_{i=1}^j (k_i-1)!,
   $$
   $k_1,\ldots,k_j\in\mathbb{N}$ with $k:=k_1+\dots+k_j>j$, where
   $[\theta]_k:=\theta(\theta+1)\dots(\theta+k-1)$.

   M\"ohle and Sagitov \cite{MS01} characterised
   exchangeable coalescents via a sequence $(F_j)_{j\in\mathbb{N}}$
   of symmetric finite measures. For each $j\in\mathbb{N}$, the measure
   $F_j$ lives on the simplex $\Delta_j:=
   \{(\zeta_1,\ldots,\zeta_j)\in [0,1]^j:\zeta_1+\dots+\zeta_j\le 1\}$ and is
   uniquely determined via its moments
   $$
   \lambda(k_1,\ldots,k_j)\ =\ \int_{\Delta_j}
   \zeta_1^{k_1-2}\cdots\zeta_j^{k_j-2}F_j(d\zeta_1,\ldots,d\zeta_j),
   \quad k_1,\ldots,k_j\ge 2.
   $$
   For the Poisson-Dirichlet coalescent, an application of Liouville's
   integration formula shows that the measure $F_j$ has density
   $f_j(\zeta_1,\ldots,\zeta_j):=\theta^j\zeta_1\cdots\zeta_j
   (1-\sum_{i=1}^j\zeta_i)^{\theta-1}$ with respect to the Lebesgue
   measure on $\Delta_j$.

   As $\Xi$ is concentrated on $\Delta^*$, it follows that
   \begin{equation} \label{finite}
      \int_\Delta\frac{|\bzeta|}{(\bzeta,\bzeta)}\,\Xi(d\bzeta)
      \ =\ \int_\Delta \frac{1}{(\bzeta,\bzeta)}\,\Xi(d\bzeta)
      \ =\ \int_{\Delta^*}\Pi_\theta(d\bzeta)\ =\ 1\ <\ \infty.
   \end{equation}
   By \cite[Proposition~29]{S00}, the Poisson-Dirichlet coalescent
   is a jump-hold Markov process with bounded transition rates and
   step function paths. By \cite[Proposition~30]{S00}, for arbitrary
   but fixed $t>0$, $\Pi_t$ does not have proper frequencies.

   The block counting process $D:=(D_t)_{t\ge 0}$, where $D_t:=|\Pi_t|$
   denotes the number of blocks of $\Pi_t$, is a decreasing process with
   rates
   $$
   g_{nk}
   \ =\ \frac{n!}{k!}\sum_{{n_1,\ldots,n_k\in\mathbb{N}}\atop{n_1+\dots+n_k=n}}
        \frac{\lambda(n_1,\ldots,n_k)}{n_1!\cdots n_k!}
   \ =\ \frac{\theta^k}{[\theta]_n}\frac{n!}{k!}
        \sum_{{n_1,\ldots,n_k\in\mathbb{N}}\atop{n_1+\dots+n_k=n}}
        \frac{1}{n_1\cdots n_k}
   \ =\ \frac{\theta^k}{[\theta]_n}s(n,k),
   $$
   $k,n\in\mathbb{N}$ with $k<n$, where the $s(n,k)$ are the absolute
   Stirling numbers of the first kind. The total rates are
   $$
   g_n\ :=\ \sum_{k=1}^{n-1}g_{nk}\ =\
   1-\frac{\theta^n}{[\theta]_n}, \quad n\in\mathbb{N}.
   $$
   Note that $g_{nk}=\mathbb{P}\{K_n=k\}$,
   $k<n$, where $K_n$ is a random variable taking values in $\{1,\ldots,n\}$
   with distribution
   $$
   \mathbb{P}\{K_n=k\}
   \ =\ \frac{\theta^k}{[\theta]_n}s(n,k),\qquad k\in\{1,\ldots,n\}.
   $$
   We have
   $$
   \gamma_n\ :=\ \sum_{k=1}^{n-1}(n-k)g_{nk}
   \ =\ \sum_{k=1}^{n-1}(n-k){\mathbb P}\{K_n=k\}
   \ =\ n-{\mathbb E}K_n\ \le\ n.
   $$
   In particular, $\sum_{n=2}^\infty\gamma_n^{-1}\ge\sum_{n=2}^\infty 1/n=\infty$.
   Together with (\ref{finite}) and $\Xi(\Delta_f)=0$, where
   $\Delta_f:=\{\bzeta\in\Delta\,|\,\zeta_1+\dots+\zeta_n=1\mbox{ for some }n\}$,
   it follows from \cite[Proposition~33]{S00} that the Poisson-Dirichlet
   coalescent stays infinite.

   If we assume no mutation, then the generator $L^\Xi$ (defined in
   Remark~\ref{remark_generator}) of the corresponding Fleming-Viot
   process reduces to
   $$
   L^{\Xi}G_f(\mu)
   \ =\ \int_{\Delta^*}\int_{E^\mathbb{N}}
   \left[
      G_f\big({\textstyle\sum_{i=1}^\infty \zeta_i \delta_{x_i}}\big)-G_f(\mu)
   \right]
   \mu^{\otimes\mathbb{N}}(d\mathbf{x})
   {\rm PD}_\theta(d\bzeta).
   $$

\bigskip

{\bf Acknowledgement.} We thank the referee for her/his careful reading
of the manuscript and helpful suggestions.

\end{document}